\documentclass[final,hidelinks,onefignum,onetabnum]{siamart220329}
\usepackage{lipsum}
\usepackage{amsfonts}
\usepackage{amssymb}
\usepackage{graphicx}
\usepackage{epstopdf}
\usepackage{subfigure}
\ifpdf
\DeclareGraphicsExtensions{.eps,.pdf,.png,.jpg}
\else
\DeclareGraphicsExtensions{.eps}
\fi

\usepackage{booktabs}
\usepackage{stmaryrd}
\SetSymbolFont{stmry}{bold}{U}{stmry}{m}{n}
\usepackage{bm}
\usepackage{caption}
\usepackage{multirow}
\usepackage{enumitem}
\setlist{leftmargin=5.5mm}
\usepackage{geometry}

\usepackage{hyperref}
\usepackage{cleveref}

\newsiamthm{assumption}{Assumption}

\renewcommand{\bf}{\textbf}

\newtheorem{thm}{Theorem}[section]

\newtheorem{lem}{Lemma}[section]

\newtheorem{rem}{Remark}[section]
\newtheorem{prop}[thm]{Proposition}

\newcommand{\bsub}{\begin{subequations}}
\newcommand{\esub}{\end{subequations}$\!$}

\newtheorem{example}{\textbf{Example}}[section]
\numberwithin{equation}{section}

\newsiamremark{remark}{Remark}
\newsiamremark{hypothesis}{Hypothesis}
\crefname{hypothesis}{Hypothesis}{Hypotheses}
\newsiamthm{claim}{Claim}
\newsiamremark{exmp}{Example}
\usepackage{algorithm}
\usepackage{algorithmic}

\makeatletter
\newcommand{\algorithmbreak}{\ifdim\pagegoal=\maxdimen \else\vfil\penalty-1000\vfilneg\fi}
\makeatother



\headers{}{}

\title{A linear, fully decoupled, and unconditionally energy-stable SAV-FEM for the Cahn--Hilliard--Navier--Stokes model
}
\author{Jinting Yang\thanks{School of Mathematics and Information Science, Henan Polytechnic University, Jiaozuo, 454003, Henan, P.R.China, \email{jtyang@hpu.edu.cn}).}
\and Nianyu Yi\thanks{Hunan Key Laboratory for Computation and Simulation in Science and Engineering; School of Mathematics and Computational Science, Xiangtan University, Xiangtan 411105, P.R.China (\email{yinianyu@xtu.edu.cn}).}
\and Peimeng Yin\thanks{Corresponding author. Department of Mathematical Sciences, The University of Texas at El Paso, El Paso, TX 79968, USA (\email{pyin@utep.edu}).}} %\footnotemark[3]

\usepackage{amsopn}

\ifpdf
\hypersetup{
  pdftitle={},
  pdfauthor={}
}
\fi

\usepackage{comment}
\usepackage{makecell}

\allowdisplaybreaks

\begin{document}

\maketitle
% REQUIRED 
\begin{abstract}
In this paper, we develop a linear, fully decoupled, and unconditionally energy-stable fully discrete finite element scheme for the Cahn--Hilliard--Navier--Stokes (CHNS) system by employing the scalar auxiliary variable (SAV) approach. 
Unlike existing SAV-based formulations that typically introduce multiple auxiliary variables or additional techniques to handle different nonlinearities, we introduce only one scalar auxiliary variable together with a novel update of the auxiliary variable to reformulate all nonlinear terms arising from the Cahn--Hilliard and Navier--Stokes equations, yielding an equivalent reformulation of the original CHNS system. 
An implicit--explicit (IMEX) Euler scheme is applied for temporal discretization, where the linear terms are treated implicitly and the nonlinear terms explicitly, while a finite element method is adopted for spatial discretization. 
The resulting fully discrete scheme can be efficiently decomposed into two linear subproblems and one scalar quadratic algebraic equation, which significantly simplifies the implementation. 
Furthermore, we prove that the proposed scheme satisfies an unconditional discrete energy dissipation law and establish its stability with respect to several relevant norms. 
Optimal-order $L^2$ error estimates are also derived for the fully discrete finite element approximation. 
Finally, a series of numerical experiments are presented to verify the theoretical results and demonstrate the efficiency of the proposed method.
\end{abstract}

	% In this paper, we carry out the optimal error estimates for the linear, decoupled, and unconditional energy-stable fully discrete finite element scheme of the Cahn-Hilliard-Navier-Stokes (CHNS) model based on scalar auxiliary variable (SAV). Firstly, the corresponding equivalent system of CHNS model with SAV is established. Secondly, the Euler implicit/explicit scheme and finite element method are used for the temporal and spatial discretizations respectively. In this way, the system is decomposed into two linear equations and a quadratic algebraic equation with one variable, which can be solved effectively. And the discrete unconditional energy dissipation, the stability of numerical solutions under different norms and the optimal error estimates for our proposed numerical schemes are also presented. Finally, we provide numerical experiments to verify our theoretical results.

% REQUIRED

%REQUIRED

% \noindent{
% \color{blue} To do list:\\
% 1. Name and affiliation. \YJ{Fixed.} \\
% 2. Full-rank computational complexity.\YJ{It can be divided into two parts: the linear part and the nonlinear part. For the full-rank format, the computational complexity of the linear part is $\mathcal{O}(\max(m,n)mn )$, and the computational complexity of the nonlinear part is $\mathcal{O}(mn)$. For the low-rank format, the computational complexity of the linear part is $\mathcal{O}(\max(m,n)(m+n)r)$, and the computational complexity of the nonlinear part is $\mathcal{O}((m+n)r^4)$.}\\
% 3. Figure 8. \YJ{Fixed.}\\
% 4. Check all the referenes. Table [34] as an example, (SAV) should be capital letter.
% \YJ{Fixed.}}
\section{Introduction}
	In this paper, we propose and analyze an SAV-based implicit--explicit (IMEX) Euler finite element method for the following CHNS system
\begin{subequations}\label{chnsequ}
	\begin{align}
		\phi_t-M\triangle \mu+\nabla\cdot(\mathbf{u}\phi) = &0,\qquad \text{in}\ \ \Omega\times(0,T],\label{equ1}\\
		\mu+\lambda\triangle\phi-g(\phi) = &0,\qquad \text{in}\ \ \Omega\times(0,T],\label{equ2}\\
		\mathbf{u}_t+(\mathbf{u}\cdot\nabla) \mathbf{u}+\nabla p-\nu\triangle \mathbf{u}-\mu\nabla\phi = &0
		,\qquad \text{in}\ \ \Omega\times(0,T],\label{equ3}\\
		\nabla\cdot \mathbf{u} = &0,\qquad \text{in}\ \ \Omega\times(0,T],\label{equ4}
	\end{align}
\end{subequations}
subject to the initial conditions
\begin{eqnarray}\label{ic}
	\mathbf{u}|_{t=0}=\mathbf{u}_0,\ \ 	\phi|_{t=0}=\phi_0,\ \  \text{in}\ \ \Omega,
\end{eqnarray}
and the boundary conditions
\begin{eqnarray}\label{bc}
	\nabla \phi\cdot\mathbf{n}=\nabla \mu\cdot\mathbf{n}=0,\ \ \mathbf{u} = \mathbf{0},\qquad \text{on}\ \ \partial\Omega\times(0,T].
\end{eqnarray}
% Here, $\Omega\subset\mathbb{R}^d,d=2,3$ is a polygonal domain with boundary $\partial\Omega$, and $g(\phi) = F^{'}(\phi)$, $F(\phi)=\frac{1}{4\varepsilon^2}(1-\phi^2)^2$ is the nonlinear free energy density with $\varepsilon$ representing the interface width, $ M, \lambda, \nu>0 $ represent the mobility constant, mixing coefficient and fluid viscosity, respectively. $\mathbf{u},p,\phi,\mu$ are the velocity, pressure, phase function and chemical potential, respectively. $\mathbf{n}$ denotes unit outward normal vector on $ \partial\Omega $. 
Here, $\Omega\subset\mathbb{R}^d$ ($d=2,3$) is a polygonal domain with boundary $\partial\Omega$, and $\mathbf{n}$ denotes the unit outward normal vector on $\partial\Omega$. The nonlinear term is given by $g(\phi)=F'(\phi)$, where
$F(\phi)=\frac{1}{4\varepsilon^2}(1-\phi^2)^2$
is the Ginzburg--Landau free-energy density, with $\varepsilon>0$ representing the interfacial width parameter. The positive constants $M$, $\lambda$, and $\nu$ denote the mobility, mixing coefficient, and fluid viscosity, respectively. The unknowns $\mathbf{u}$, $p$, $\phi$, and $\mu$ represent the velocity field, pressure, phase variable, and chemical potential, respectively.

% We can easily get the following
% energy dissipation law of the initial boundary problem \eqref{chnsequ}-\eqref{bc},
% \begin{eqnarray}\label{en}
% 	\frac{dE(\phi,\textbf{u})}{dt}=-M\|\nabla\mu\|_0^2-\nu\|\nabla \textbf{u}\|_0^2,
% \end{eqnarray}
% where $E(\phi,\textbf{u})=\int_\Omega\{\frac{1}{2}(|\textbf{u}|^2+|\nabla \phi|^2)+F(\phi)\}dx$ is the total energy, and $ \|\cdot\|_0 $ is defined as $ L^2$ norm.
% For the convenience of the following theoretical analysis, we set the parameters $ M=\lambda=\nu=1 $.

By taking suitable inner products of \eqref{chnsequ} and applying the boundary conditions \eqref{bc}, one readily obtains the following energy dissipation law for the initial-boundary value problem \eqref{chnsequ}--\eqref{bc}:
\begin{eqnarray*}
	\frac{dE(\phi,\textbf{u})}{dt}=-M\|\nabla\mu\|_0^2-\nu\|\nabla \textbf{u}\|_0^2,
\end{eqnarray*}
where $E(\phi,\textbf{u})=\int_\Omega\{\frac{1}{2}(|\textbf{u}|^2+|\nabla \phi|^2)+F(\phi)\}dx$ 
denotes the total free energy, and $\|\cdot\|_0$ denotes the $L^2(\Omega)$ norm for scalar functions and vector fields, with the squared norm defined by integrating the sum of squares of all components over $\Omega$. 
For convenience in the subsequent analysis, we assume throughout the paper that
$ M=\lambda=\nu=1 $.
%Under this assumption, the energy law \eqref{en} takes the simplified form stated above.

As a coupled nonlinear system, \eqref{chnsequ}--\eqref{bc} presents several challenges in the design of effective numerical schemes:
(i) the nonlinear coupling between the velocity and phase variables through the convective term in the phase equation and the stress term in the momentum equation;
(ii) the nonlinear energy potential and the stiffness induced by the small interfacial width parameter $\varepsilon$ in the phase equation;
(iii) the requirement that the resulting discrete scheme, even after decoupling the variables, preserves the energy dissipation law.

Constructing efficient numerical schemes for the CHNS system is highly challenging, particularly when one seeks to preserve a discrete energy dissipation law at the numerical level. Over the past two decades, a variety of numerical methods have been developed for the CHNS system. Among the most widely used approaches for temporal discretization are the linear stabilization method \cite{CCS2017,LQT2016,SY2010,SY2015,YY2017}, the convex-splitting method \cite{HW2015,HWW2009,SWWW2012,WWL2009}, the invariant energy quadratization (IEQ) method \cite{CLYY2025,CLYY2026,YZH2018,YZW2017,ZWY2017}, and the scalar auxiliary variable (SAV) method \cite{SX2018,SXY2018,YYC2024,YY2024,ZY2022}.
Shen and Yang \cite{SY2011} developed several effective time-discrete schemes for two-phase incompressible flows with variable densities and viscosities, and established the corresponding discrete energy dissipation laws. 
%Subsequently, they proposed two classes of decoupled, unconditionally energy-stable schemes for the CHNS system based on stabilization and convex-splitting techniques. At each time step, these schemes require only the solution of a sequence of elliptic problems, including a pressure Poisson equation \cite{SY2015}. 
Combining the convex-splitting method for the Cahn--Hilliard equation with a pressure-projection method for the Navier--Stokes equations, Han and Wang \cite{HW2015} constructed a second-order time-stepping scheme for the CHNS system.
For two-phase incompressible flow models with variable density and viscosity, Chen and Yang \cite{CY2022} proposed a fully decoupled numerical approach that combines Strang operator splitting with a decoupling strategy based on the zero-energy-contribution property. More recently, first- and second-order unconditionally energy-stable finite element schemes for the CHNS system were developed in \cite{CLYY2026} by combining the IEQ approach with finite element discretization. The resulting schemes were shown to preserve mass conservation and energy dissipation, and the theoretical findings were validated through a series of numerical experiments.

On the other hand, rigorous error analysis for numerical schemes of the CHNS system is considerably more challenging. Cai \textit{et al.} \cite{CCS2017,CS2018} established error estimates for both semi-discrete and fully discrete energy-stable schemes for two-phase incompressible flows. Li and Shen \cite{LS20201} developed a second-order, weakly coupled, linear, and energy-stable SAV-MAC scheme for the CHNS equations, and proved convergence results for the corresponding Cahn--Hilliard--Stokes system. For further theoretical developments on error analysis of numerical schemes for CHNS-type models, we refer the reader to \cite{CSWY2023,CJLWW2024,CJWW2022,G2025}.
%Recently, Diegel, Feng, and Wise \cite{DFW2015} proposed a fully discrete convex-splitting finite element scheme for the Cahn--Hilliard--Darcy--Stokes system based on the backward Euler method. They proved the unconditional energy stability of the scheme and derived error estimates for the phase-field variable, chemical potential, and velocity. However, no convergence result for the pressure approximation was established.
Subsequently, Diegel, Wang, Wang and Wise \cite{DWWW2017} extended this framework to the Cahn–Hilliard–Navier–Stokes system with a second-order in time scheme combining Crank–Nicolson and Adams–Bashforth extrapolation, and established unconditional energy stability and optimal convergence rates.
To the best of our knowledge, Cai \textit{et al.} \cite{CSWY2023} were the first to establish optimal $L^2$-error estimates for a convex-splitting finite element scheme for the CHNS system. By introducing a novel projection operator and employing negative-norm superconvergence techniques, they proved that, in both the Taylor--Hood and MINI finite element spaces, the phase-field variable, chemical potential, and pressure achieve optimal-order convergence in the $L^2$ norm for $r\ge 1$, while the velocity attains optimal-order $L^2$ convergence for $r\ge 2$. The theoretical predictions were further confirmed by numerical experiments.
More recently, a fully discrete finite element scheme based on a double scalar auxiliary variable (SAV) approach combined with a pressure-correction projection method was proposed in \cite{G2025}. The resulting scheme is linear, fully decoupled, and unconditionally energy stable. Furthermore, optimal-order $L^2$ error estimates were rigorously established in the $P_r\times P_r\times P_{r+1}\times P_r$ finite element space without relying on any specially designed projection operator.

In this paper, we propose a linear, fully decoupled, and unconditionally energy-stable fully discrete finite element scheme for the CHNS system based on the SAV approach.
% The SAV method was proposed in \cite{SX2018,SXY2018}. Its core idea is to introduce a single scalar auxiliary variable to represent the discrete energy, to handle gradient flow systems with nonlinear potential functions. This approach transforms the original nonlinear system into a linear one and preserves a discrete energy dissipation law at the fully discrete level. However, when extending the SAV method to the coupled CHNS system, the problem becomes more challenging due to additional nonlinearities (such as the convection term) and the incompressibility constraint introduced by the Navier–Stokes (NS) equations.
The SAV method, originally introduced in \cite{SX2018,SXY2018}, provides an efficient framework for the numerical approximation of gradient flow systems with nonlinear free-energy potentials. Its key idea is to introduce a scalar auxiliary variable associated with the nonlinear energy functional, thereby transforming the original nonlinear system into a linear one while preserving a discrete energy dissipation law at the fully discrete level.

In the literature, existing SAV-based schemes for the CHNS system have been developed to treat the nonlinearities arising from the Cahn--Hilliard and Navier--Stokes components separately.
They employed either the multi-scalar auxiliary variable (MSAV) approach \cite{G2025,LS2020} or a combination of the scalar auxiliary variable (SAV) and zero-energy-contribution (ZEC) techniques \cite{CY2021,Y20210}. In the MSAV framework, multiple independent scalar auxiliary variables are introduced to facilitate the decoupling of different nonlinear components. For the CHNS system, two auxiliary variables are typically employed to treat the nonlinear terms arising from the Cahn--Hilliard and Navier--Stokes equations separately, resulting in two independent scalar ordinary differential equations and enabling the development of linear, decoupled, and unconditionally energy-stable schemes.
The SAV-ZEC approach also introduces two auxiliary variables, but they serve different purposes. One auxiliary variable is associated with the nonlinear free-energy potential and contributes directly to the discrete energy, while the other is introduced to handle the nonlinear convection and coupling terms whose overall contribution to the energy dissipation law is zero. By exploiting the zero-energy-contribution property of these terms, the SAV-ZEC approach achieves linearization and unconditional energy stability while preserving a decoupled and efficient numerical structure.

Although both the MSAV and SAV-ZEC approaches provide effective decoupling strategies and lead to unconditionally energy-stable schemes, they require the introduction of multiple auxiliary variables. As a result, the nonlinear system needs to be decomposed multiple times, which increases the complexity of both the numerical formulation and implementation and may introduce additional computational costs.

%In contrast, the present work employs only a single scalar auxiliary variable to deal all the nonlinear terms simultaneously, combined with an implicit--explicit (IMEX) time discretization and a finite element spatial discretization.
In contrast, the present work introduces only a single scalar auxiliary variable to simultaneously reformulate all nonlinear terms, together with a novel update strategy for the scalar auxiliary variable. An implicit--explicit (IMEX) Euler time discretization is then employed, combined with a finite element spatial discretization.
The strategy of employing a unified technique to handle multiple nonlinear terms has become an effective strategy for solving large-scale coupled systems \cite{liu2026structure, li2026fully}.
Moreover, unlike the pressure-correction SAV scheme proposed in \cite{YY2024}, where the velocity and pressure variables are decoupled through a projection step, our approach preserves the velocity--pressure coupling by directly solving the Navier--Stokes subsystem at each time step. This choice is motivated by both analytical and computational considerations: it simplifies the error analysis and allows us to establish optimal error estimates under the weaker regularity assumption $\phi\in H^2(\Omega)$, rather than requiring $\phi\in H^3(\Omega)$.

The proposed method remains computationally efficient, as each time step only involves solving one linear Cahn--Hilliard-type system, one Stokes-type system, and one scalar quadratic equation for the auxiliary variable. Furthermore, we prove that the proposed single-SAV IMEX finite element scheme is uniquely solvable and satisfies an unconditional discrete energy dissipation law. We also establish stability and derive optimal-order $L^2$ error estimates for the fully discrete finite element approximation. 
The error analysis is developed by decomposing the total error into a projection error and a discrete error using the standard $L^2$ projection and the Ritz projection. A major challenge is that the nonlinear potential $g(\phi)=F'(\phi)$ is not globally Lipschitz. This difficulty is resolved by proving the uniform $L^\infty$ boundedness of the numerical solution, allowing the error estimates to be established without requiring any restrictive assumptions on the free energy.

The remainder of this paper is organized as follows. In Section \ref{secScheme}, we present the model and its equivalent form. Then we develop the decoupled scheme and show its associated  stability results. In Section \ref{secError}, we establish the optimal error estimates  for the solutions of a fully discrete system. Finally, several numerical examples are presented to test and verify the theoretical results of the proposed numerical scheme in Section \ref{secNum}.

\section{The fully discrete SAV-FEM scheme}\label{secScheme}
% 	We introduce a scalar energy variable
% $\tilde E_1[\phi]=\int_\Omega F(\phi)dx>-c_0 (c_0>0)$ and let $\tilde C_0>c_0$ such that $\tilde E_1[\phi]+\tilde C_0>0$. We further denote $E_1=\tilde E_1+C_0$. In this case, $ E_1 $ has a positive lower bound $C_0=\tilde C_0-c_0$. Let $r(t)=\sqrt{E_1[\phi]}$,  the CHNS system \eqref{chnsequ} can be transformed into
\subsection{Prerequisites and the fully discrete scheme}

We first introduce the scalar energy functional
$
\tilde E_1[\phi]=\int_\Omega F(\phi)\,dx$, $\tilde E_1[\phi] > -c_0$ with $c_0>0$,
and choose a constant $\tilde C_0>c_0$ such that $\tilde E_1[\phi]+\tilde C_0>0$. Define
\begin{equation*}\label{C0}
    E_1[\phi]=\tilde E_1[\phi]+C_0,\qquad C_0=\tilde C_0-c_0>0,
\end{equation*}
so that $E_1[\phi]\ge C_0>0$. We then introduce the scalar auxiliary variable
\[
r(t)=\sqrt{E_1[\phi(t)]}.
\]
With this reformulation, the CHNS system \eqref{chnsequ} can be reformulated as
\begin{subequations}\label{1sequ}
	\begin{align}
		\phi_t-\triangle \mu+\frac{r(t)}{\sqrt{E_1[\phi]}}\nabla\cdot(\textbf{u}\phi)=0,&\qquad \text{in}\ \ \Omega\times(0,T],\label{1sequ1}\\
		\mu+\triangle\phi-g(\phi)=0,&\qquad \text{in}\ \ \Omega\times(0,T],\label{1sequ2}\\
		\textbf{u}_t+\frac{r(t)}{\sqrt{E_1[\phi]}}(\textbf{u}\cdot\nabla) \textbf{u}+\nabla p-\triangle \textbf{u}-\frac{r(t)}{\sqrt{E_1[\phi]}}\mu\nabla\phi
		=0,&\qquad \text{in}\ \ \Omega\times(0,T],\label{1sequ3}\\
		\nabla\cdot \textbf{u}=0,&\qquad \text{in}\ \ \Omega\times(0,T],\label{1sequ4}\\
		r_t=\frac{1}{2r(t)}\Big{(} \int_\Omega g(\phi)\phi_tdx+\frac{r(t)}{\sqrt{E_1[\phi]}}((\textbf{u}\cdot\nabla \textbf{u},\textbf{u})\nonumber\\
		+(\textbf{u}\cdot\nabla\phi,\mu)-(\mu\nabla\phi,\textbf{u}))\Big{)},&\qquad \text{in}\ \ \Omega\times(0,T] \label{1sequ5}.
	\end{align}
\end{subequations}

Unlike the standard SAV approach, the update equation for $r$ in \eqref{1sequ5} contains additional correction terms. These additional terms are essential for enabling a single scalar auxiliary variable to simultaneously treat multiple nonlinear terms while maintaining the key structural properties of the original system.

% For any integer $k \geq 0$ and real number $1 \leq p \leq \infty$, we denote $L^p(\Omega)$ and $W^{k,p}(\Omega)$ respectively the Lebesgue
% spaces and Sobolev spaces of functions defined on $\Omega$ with the abbreviation $H^k(\Omega)=W^{k,2}(\Omega)$. Let $\textbf{H}^{k}(\Omega)=(H^k(\Omega))^2$, $\textbf{L}^{p}(\Omega)=(L^p(\Omega))^2$ are the corresponding vector space and $\textbf{H}_0^1(\Omega)$ is a closed subspace of $\textbf{H}^1(\Omega)$ consisting of the functions with zero trace on $\Omega$. We equip the $H^k(\Omega)$ space with the norm $\|\cdot\|_k$ and set
For any integer $k \ge 0$ and real number $1 \le p \le \infty$, we denote by
$L^p(\Omega)$ and $W^{k,p}(\Omega)$ the Lebesgue and Sobolev spaces of functions
defined on $\Omega$, respectively. We set $H^k(\Omega)=W^{k,2}(\Omega)$.
Let $\mathbf{L}^p(\Omega) = (L^p(\Omega))^d$ and $\mathbf{H}^k(\Omega) = (H^k(\Omega))^d$ 
be the corresponding vector-valued spaces. Moreover, $\mathbf{H}_0^1(\Omega)$
denotes the closed subspace of $\mathbf{H}^1(\Omega)$ consisting of functions
with vanishing trace on $\partial\Omega$.
We endow $H^k(\Omega)$ with the norm $\|\cdot\|_k$ and introduce the following spaces:
\begin{eqnarray*}
	&&W=H^1(\Omega),\ \ \textbf{X}=\textbf{H}_0^1(\Omega),\ \ \textbf{Y}=\textbf{L}^2(\Omega),\\
	&&M=L_0^2(\Omega)=\left\lbrace q\in L^2(\Omega);\int_{\Omega}qdx=0\right\rbrace ,
	\\
	&&\textbf{H}=\{\textbf{v}\in \textbf{Y};\nabla\cdot \textbf{v}=0\  \text{in}\  \Omega, \ \ \text{and}\ \ \textbf{v}\cdot \textbf{n}|_{\partial\Omega}=0\},\\
	&&\textbf{V}=\{\textbf{v}\in \textbf{X};\nabla\cdot \textbf{v}=0\},\ \ \
	\mathbb{R}=\{ \text{the\ space\ of\ real\ numbers}\}.
\end{eqnarray*}
For all $\textbf{u}\in\textbf{V}$ and $\textbf{v},\textbf{w}\in\textbf{X}$, it follows
\begin{eqnarray*}
	((\textbf{u}\cdot\nabla)\textbf{v},\textbf{w})=((\textbf{u}\cdot\nabla)\textbf{v},\textbf{w})+\frac{1}{2}((\nabla\cdot\textbf{u})\textbf{v},\textbf{w})=\frac{1}{2}((\textbf{u}\cdot\nabla)\textbf{v},\textbf{w})-\frac{1}{2}((\textbf{u}\cdot\nabla)\textbf{w},\textbf{v}).
\end{eqnarray*}
In particular, taking $\mathbf{w}=\mathbf{v}$ yields
\begin{equation}\label{triliform}
((\mathbf{u}\cdot\nabla)\mathbf{v},\mathbf{v})=0.
\end{equation}

Taking the inner products of \eqref{1sequ1}--\eqref{1sequ4} with $\mu$, $\phi_t$, $\mathbf{u}$, and $p$, respectively, and multiplying \eqref{1sequ5} by $2r$, then using \eqref{triliform} together with the identity
\begin{eqnarray}\label{ab}
	2(a-b,a)=|a|^2-|b|^2+|a-b|^2, \qquad \forall a, b\in R,
\end{eqnarray} 
we obtain the following energy dissipation law:
\begin{eqnarray*}
	\frac{d\tilde{E}(\phi,\textbf{u},r)}{dt}=-\|\nabla\mu\|_0^2-\|\nabla \textbf{u}\|_0^2 \leq 0,
\end{eqnarray*}
where $\tilde{E}(\phi,\textbf{u},r)=\int_\Omega\frac{1}{2}(|\textbf{u}|^2+|\nabla\phi|^2)dx+r^2$ denotes the modified energy.

% In order to establish the stability and the convergence of the SAV schemes for the CHNS system \eqref{chnsequ}, some Sobolev inequalities are needed \cite{A1975, FHL2007, H2003, H2005, HS2007}
% To establish the stability and convergence of the SAV schemes for the CHNS system \eqref{1sequ}, we will repeatedly use the following Sobolev inequalities (see, e.g., \cite{A1975,FHL2007,H2003,H2005,HS2007}):
In order to establish the stability and convergence of the SAV schemes for the CHNS system \eqref{1sequ}, we shall frequently make use of the following Sobolev inequalities (see, e.g., \cite{A1975,FHL2007,H2003,H2005,HS2007}):
\begin{align}\label{s1}
	&\|\textbf{v}\|_{L^r}\leq C\|\nabla \textbf{v}\|_0\quad (2\leq r\leq6),\ \ \ \ \ 
	\|\textbf{v}\|_{L^4}\leq C\|\textbf{v}\|_0^{\frac{1}{2}}\|\nabla \textbf{v}\|_0^{\frac{1}{2}},\ \ \
	%\|v\|_{L^4}\leq C\|v\|_0^{\frac{1}{2}}\|\nabla v\|_0^{\frac{1}{2}},\ \ \
	\forall\ \textbf{v}\in \textbf{X},\\
	\label{s2}
	&\|\textbf{v}\|_{L^\infty}\leq C\|\textbf{v}\|_0^{\frac{1}{2}}\|A \textbf{v}\|_0^{\frac{1}{2}}, \ \ \  \ \ \ \ \ \ \ \ \ \  \ \ \  \|\textbf{v}\|_2\leq C\|A\textbf{v}\|_0,\ \ \
	\forall\ \textbf{v}\in \textbf{H}^2(\Omega)\cap \textbf{V},\\
	\label{s3}
	&\|\phi\|_{L^r}\leq C\|\phi\|_1\quad (2\leq r\leq \infty),\ \ \ \ \ \ \  
	\forall \phi\in W,\\
	\label{s4}
	&\|\phi\|_{L^\infty}\leq C\|\phi\|_0^{\frac{1}{2}}(\|\phi\|_0^2+\|\triangle\phi\|_0^2)^{\frac{1}{4}},\ \ \  
	\forall\ \phi\in H^2(\Omega)\  \text{with}\ \frac{\partial\phi}{\partial n} \Big{|} _{\partial\Omega}=0.
\end{align}
Here, $ A=-P\triangle $ denotes the Stokes operator, where $P$ is the $L^2$-orthogonal projection from $\mathbf{Y}$ onto $\mathbf{H}$. Throughout this paper, $C>0$ denotes a generic constant depending only on the domain $\Omega$. The value of $C$ may vary from one estimate to another.
%is Stokes operator with $ P $ is the $L^2$-orthogonal projection from $\textbf{Y}$ to $\textbf{H}$. $C>0$ is a general constant that depends on $ \Omega $. In the following analysis, C will take different values according to different situations.

% We will use the classical discrete Gronwall's Lemma frequently in the following proofs.
We will frequently use the following classical discrete Gronwall's Lemma in the subsequent analysis.
\begin{lem} \cite{S1990,S1992} For integers $n\geq0$, let $c,a_n,b_n,c_n,d_n\geq0$ such
	that
	\begin{eqnarray}\label{2.1}
		a_m+\Delta t\sum_{n=0}^{m}b_n\leq\Delta t\sum_{n=0}^{m}d_na_n+\Delta t\sum_{n=0}^{m}c_n+c,\ \ \forall\ m\geq 0.
	\end{eqnarray}
	Suppose further that $ \Delta td_n<1 $, $\forall n \geq 0$, and define $ \sigma_n=(1-\Delta td_n)^{-1} $. Then
	\begin{eqnarray}\label{2.2}
		a_m+\Delta t\sum_{n=0}^{m}b_n\leq e^{\left(\Delta t\sum\limits_{n=0}^{m}d_n\sigma_n\right) }\left( \Delta t\sum\limits_{n=0}^{m}c_n+c\right) ,\ \ \forall\ m\geq 0.
	\end{eqnarray}
\end{lem}
\begin{rem} 
If the first summation on the right-hand side of \eqref{2.1} is taken from $n=0$ to $m-1$, namely,
%If the first sum on the right in \eqref{2.1} extends only up to $ m-1 $, 
then the estimate \eqref{2.2} holds for all $ n > 0 $ with
	$ \sigma_n = 1 $. 
\end{rem}

We consider the fully discrete finite element method based on the IMEX Euler scheme for solving the model system \eqref{1sequ}. Denote
\begin{eqnarray*}
	\Delta t=T/N,\ \ t^n=n\Delta t,\ \ d_tf^{n+1}=\frac{f^{n+1}-f^n}{\Delta t},\ \ \text{for}\ n\leq N,
\end{eqnarray*}
where $T$ denotes the final time, $N$ is the number of time steps, and $\Delta t$ is the uniform time-step size. Let $\Omega\subset\mathbb{R}^d$ be a polygonal domain, and let $\mathcal{T}_h$ be a shape-regular triangulation of $\Omega$ consisting of triangular elements $K$. We denote $h=\max_{K\in\mathcal {T}_h}h_K$, where $h_K$ is the diameter of the element $ K $. 

%We construct the phase field-velocity-pressure finite element spaces: 
%$(W_h\times \textbf{X}_h\times M_h)\subset (W\times \textbf{X}\times M)$ based on the mesh $\mathcal {T}_h$ 
% For any non-negative integer $ l $, we denote $P_l(K)$ as the space of polynomials over $ K $ of degree less than or equal to $ l $. Introduce $\hat{b}\in H_0^1(K)$ to take the value $ 1 $ at the barycenter of $ K $ and such that $0\leq\hat{b}(x)\leq1$, which is called the bubble function. Choosing the following phase field-velocity-pressure finite element spaces

For any non-negative integer $l$, let $P_l(K)$ denote the space of polynomials on $K$ of degree less than or equal to $l$. Let $\hat{b}\in H_0^1(K)$ be the bubble function satisfying $\hat{b}(x_K)=1$ at the barycenter $x_K$ of $K$ and $0\le \hat{b}(x)\le 1$ for all $x\in K$.
We define the phase-field, velocity, and pressure finite element spaces as follows:
\begin{eqnarray*}
	&&W_h=\{\phi_h\in C^0(\Omega)\cap W;\ \phi_h|_K\in P_1(K),\ \forall\ K\in\mathcal{T}_h\},\\
	&&\textbf{X}_h=\{\textbf{v}_h\in \textbf{C}^0(\Omega)\cap \textbf{X};\ \textbf{v}_h|_K\in \left( P_1(K)\oplus\ \text{span}\{\hat{b}\}\right) ^d,\ \forall\ K\in\mathcal{T}_h\},\\
	&&M_h=\{q_h\in C^0(\Omega)\cap M;\ q_h|_K\in P_1(K),\ \forall\ K\in\mathcal{T}_h\}.
\end{eqnarray*}
Define the subspace $\textbf{V}_h$ of $\textbf{X}_h$ as
\begin{eqnarray*}
	\textbf{V}_h=\{\textbf{v}_h\in \textbf{X}_h;\ \nabla\cdot \textbf{v}_h=0\}.
\end{eqnarray*}

% The so-called inf-sup inequality is defined by: for each $q_h\in M_h$, there exists $\textbf{0}\neq \textbf{v}_h\in \textbf{X}_h$ such that
% \begin{eqnarray}\label{inf-sup}
% 	(q_h,\nabla\cdot \textbf{v}_h)\geq\beta\|q_h\|_0\|\textbf{v}_h\|_1,
% \end{eqnarray}
% where $\beta>0$ is a constant independent of $ h $ and $\Delta t$.

The inf-sup condition is stated as follows: for each $q_h \in M_h$, there exists a nonzero $\mathbf{v}_h \in \mathbf{X}_h$ such that
\begin{equation}\label{inf-sup}
(q_h, \nabla \cdot \mathbf{v}_h)
\ge \beta |q_h|_0 |\mathbf{v}_h|_1,
\end{equation}
where $\beta>0$ is a constant independent of $h$ and $\Delta t$.

% Set $P_h:\textbf{Y}\rightarrow \textbf{X}_h$ as the $L^2$-orthogonal projection \cite{AM1994,H2015,HS2008} defined by
Let $P_h:\mathbf{Y}\to \mathbf{X}_h$ be the $L^2$-orthogonal projection \cite{AM1994,H2015,HS2008} defined by
\begin{eqnarray}\label{orth1}
	(P_h\textbf{v},\textbf{v}_h)=(\textbf{v},\textbf{v}_h),\ \ \forall\ \textbf{v}\in \textbf{Y},\ \textbf{v}_h\in \textbf{X}_h.
\end{eqnarray}
The projection satisfies the following inequalities
\begin{eqnarray}\label{orth2}
	\begin{cases}
		\|P_h\textbf{v}\|_1\leq C\|\textbf{v}\|_1,\ \ \|\textbf{v}-P_h\textbf{v}\|_0\leq Ch\|\textbf{v}\|_1,\ \ \forall\ \textbf{v}\in \textbf{X},\\
		\|\textbf{v}-P_h\textbf{v}\|_0+h\|\textbf{v}-P_h\textbf{v}\|_1\leq Ch^2\|A\textbf{v}\|_0,\ \ \ \forall\ \textbf{v}\in \textbf{H}^2(\Omega)\cap\textbf{V}.
	\end{cases}
\end{eqnarray}

Denote the discrete Stokes operator $A_h=-P_h\triangle_h:\textbf{X}_h\rightarrow \textbf{X}_h$ by the definition:
\begin{equation*}
(A_h \textbf{u}_h, \textbf{v}_h) = (A_h^{1/2} \textbf{u}_h, A_h^{1/2} \textbf{v}_h) = (\nabla \textbf{u}_h, \nabla \textbf{v}_h), \ \ \forall\ \textbf{u}_h,\textbf{v}_h\in \textbf{X}_h,
\end{equation*}
and the discrete Laplacian operator $ \triangle_h $ is defined by 
\begin{eqnarray}\label{Ah-1}
	(-\triangle_hw_h,s_h)&=&(\nabla w_h,\nabla s_h),\ \ \forall\ w_h,s_h\in W_h(or\ \textbf{X}_h).
\end{eqnarray}
%Here, the discrete norm can be defined, where
We also introduce the following discrete norms. For $\mathbf{v}_h \in \mathbf{V}_h$, define
\begin{eqnarray}
	\|\textbf{v}_h\|_2&=&\|A_h\textbf{v}_h\|_0,\ \ \|\textbf{v}_h\|_{-1}=\|A_h^{-1/2}\textbf{v}_h\|_0,\\
	\label{Ah-11}
	\|A_h^{1/2}\textbf{v}_h\|_0&=&\|\nabla\textbf{v}_h\|_0,\ \ \|\nabla A_h^{-1/2}\textbf{v}_h\|_0=\|\textbf{v}_h\|_0.
\end{eqnarray}

Let $R_h:H^1(\Omega)\rightarrow W_h$ be the $Ritz$ orthogonal projection \cite{CMS2020} in the finite element space
\begin{eqnarray}\label{orth3}
	(\nabla R_h\phi,\nabla\psi_h)=(\nabla \phi,\nabla\psi_h),\ \ \forall\ \phi\in H^1,\ \psi_h\in W_h.
\end{eqnarray}
The projection satisfies the following inequality
\begin{eqnarray}\label{orth4}
	\|\phi-R_h\phi\|_0+h\|\nabla(\phi-R_h\phi)\|_0\leq Ch^2\|\phi\|_2.
\end{eqnarray}

%	\begin{lem} \label{lem5.12} \cite{T2006}
	%		Let $R_h\psi$ be the $Ritz$ projection of $ \psi $. It holds that
	%		\begin{eqnarray*}
		%			\|\psi-R_h\psi\|_{-s}\leq Ch^{s+q}\|\psi\|_q,\ \ \ \text{for}\ 0\leq s\leq l-1,\ \ 1\leq q\leq l+1.
		%		\end{eqnarray*}
	%	\end{lem}

\begin{lem} \label{lem5.11} \cite{CMS2020}
	For all $\phi_h\in W_h$, the discrete Laplacian operator $\triangle_h$ satisfies the following estimates:
	\begin{eqnarray*}
		\|\triangle_h\phi_h\|_{-2}&\leq& C\|\phi_h\|_0,\\
		\|\triangle_h \phi\|_{-1}&\leq& C\|\nabla \phi\|_0.
	\end{eqnarray*}
\end{lem}

The fully discrete SAV-FEM scheme based on the formulation \eqref{1sequ} is formulated as follows: find $(\phi_h^{n+1},\mu_h^{n+1},\textbf{u}_h^{n+1},p_h^{n+1},r_h^{n+1})\in W_h\times W_h\times \textbf{X}_h\times M_h\times R$ such that, for all $(\psi_h,\tau_h,\textbf{v}_h,q_h)\in W_h\times W_h\times \textbf{X}_h\times M_h$, it holds
\begin{subequations}\label{dsequ11}
	\begin{align}
		(d_t\phi^{n+1}_h,\psi_h)&+(\nabla \mu_h^{n+1},\nabla \psi_h)+\frac{r_h^n}{\sqrt{E_{1,h}^n}}(\textbf{u}_h^n\cdot\nabla\phi_h^n,\psi_h)=0,\\
		(\mu_h^{n+1},\tau_h)&-(\nabla\phi_h^{n+1},\nabla\tau_h)-(g(\phi_h^n),\tau_h)=0,\\
		(d_t\textbf{u}_h^{n+1},\textbf{v}_h)&+\frac{r_h^n}{\sqrt{E_{1,h}^n}}(\textbf{u}_h^n\cdot\nabla \textbf{u}_h^n,\textbf{v}_h)-(p_h^{n+1},\nabla \cdot \textbf{v}_h)+(\nabla \textbf{u}_h^{n+1},\nabla \textbf{v}_h)\nonumber\\
		&\ \ \ -\frac{r_h^n}{\sqrt{E_{1,h}^n}}(\mu^n_h\nabla\phi_h^n,\textbf{v}_h)
		=0,\\
		(\nabla\cdot \textbf{u}_h^{n+1},q_h)&=0,\\
		d_tr_h^{n+1}&=\frac{1}{2r_h^{n+1}}\Big{(} (g(\phi_h^n),d_t\phi_h^{n+1})+\frac{r_h^n}{\sqrt{E_{1,h}^n}}((\textbf{u}_h^n\cdot\nabla \textbf{u}_h^n,\textbf{u}_h^{n+1}) \nonumber\\
		&\ \ \ +(\textbf{u}_h^n\cdot\nabla\phi_h^n,\mu_h^{n+1})-(\mu_h^n\nabla\phi_h^n,\textbf{u}_h^{n+1}))\Big{)},
	\end{align}
\end{subequations}
with 
\[\phi_h^0=R_h\phi^0,\quad  \mu_h^0=R_h\mu^0, \quad  \textbf{u}_h^0=P_h\textbf{u}^0,\quad   r_h^0=\sqrt{E_{1,h}^0}=\sqrt{E_1[\phi_h^0]},\]  and 
\[E_{1,h}^n=E_1[\phi_h^n]=\int_\Omega F(\phi_h^n)dx.\]

\begin{thm}\label{thm11}
	Scheme \eqref{dsequ11} is unconditionally energy stable in the sense that
	\begin{eqnarray}\label{ener}
		\tilde{E}_h^{n+1}(\phi,\textbf{u},r)-\tilde{E}_h^n(\phi,\textbf{u},r)\leq-\Delta t(\|\nabla\mu_h^{n+1}\|_0^2+\|\nabla \textbf{u}_h^{n+1}\|_0^2),
	\end{eqnarray}
	where $\tilde{E}_h^n(\phi,\textbf{u},r)=\int_\Omega\frac{1}{2}(|\textbf{u}_h^n|^2+|\nabla \phi_h^n|^2)dx+|r_h^n|^2$.
\end{thm}
\begin{proof}
	%Taking inner product with $(\mu^{n+1},d_t\phi^{n+1},\textbf{u}^{n+1},p^{n+1})$ in \eqref{dsequ1}-\eqref{dsequ4} respectively, and multiplying \eqref{dsequ5} by $2r^{n+1}$, using \eqref{triliform} and \eqref{ab}, we have
    Taking the inner products of (\ref{dsequ11}a)--(\ref{dsequ11}d) with $\mu_h^{n+1}$, $d_t\phi_h^{n+1}$, $\mathbf{u}_h^{n+1}$, and $p_h^{n+1}$, respectively, and multiplying (\ref{dsequ11}e) by $2r_h^{n+1}$, then applying \eqref{triliform} and the identity \eqref{ab}, we obtain
	\begin{eqnarray*}
		&&\frac{1}{2}(\|\textbf{u}_h^{n+1}\|_0^2-\|\textbf{u}_h^n\|_0^2+\|\textbf{u}_h^{n+1}-\textbf{u}_h^n\|_0^2)+\frac{1}{2}(\|\nabla\phi_h^{n+1}\|_0^2-\|\nabla\phi_h^n\|_0^2+\|\nabla(\phi_h^{n+1}-\phi_h^n)\|_0^2)\nonumber\\
		&&+|r_h^{n+1}|^2-|r_h^n|^2+|r_h^{n+1}-r_h^n|^2\nonumber\\
		&=&-\Delta t(\|\nabla\mu_h^{n+1}\|_0^2+\|\nabla \textbf{u}_h^{n+1}\|_0^2),
	\end{eqnarray*}
	which yields \eqref{ener}.
\end{proof}

% \begin{rem}
% It is worth noting that the term $\frac{r(t)}{\sqrt{E_1[\phi]}}$ added in front of the nonlinear coupling term in equation \eqref{1sequ} can also be used as the coefficient of the nonlinear term $g(\phi)$. In this way, the corresponding term $g(\phi)$ in equation (\ref{1sequ}e) also needs to be multiplied by this auxiliary variable. that is (2.1b) and (2.1e) changes to 
% \PY{Show the two equations}
% In addition, Theorem \ref{thm11} still holds for the corresponding fully discrete SAV-FEM scheme.
% However, to simplify the subsequent theoretical analysis, this modification scheme is not adopted in this paper. Instead, in the fully discrete scheme \eqref{dsequ11}, $g(\phi)$ is directly and explicitly handled.
% \end{rem}

\begin{rem}
The scalar coefficient $\frac{r(t)}{\sqrt{E_1[\phi]}}$
introduced in \eqref{1sequ} is not restricted to the nonlinear coupling terms and may also be incorporated into the nonlinear potential term $g(\phi)$. In this case, equations \eqref{1sequ2} and \eqref{1sequ5} become
\begin{align}
& \mu+\Delta\phi-\frac{r(t)}{\sqrt{E_1[\phi]}}g(\phi)=0,
\label{rem:eq1}\\
& r_t
=\frac{1}{2r(t)}
\left(
\frac{r(t)}{\sqrt{E_1[\phi]}}
\int_\Omega g(\phi)\phi_t\,dx
+\frac{r(t)}{\sqrt{E_1[\phi]}}
\bigl(
(\mathbf u\cdot\nabla\mathbf u,\mathbf u)
+(\mathbf u\cdot\nabla\phi,\mu)
-(\mu\nabla\phi,\mathbf u)
\bigr)
\right).
\label{rem:eq2}
\end{align}
The corresponding fully discrete SAV-FEM scheme remains linear, fully decoupled, and unconditionally energy stable. Furthermore, the same optimal-order error estimates established in Section~\ref{secError} can be obtained for this alternative formulation. However, to simplify the theoretical analysis and computational realization of the proposed SAV-FEM framework, we do not pursue this formulation in the present work. Instead, the nonlinear potential term $g(\phi)$ is treated explicitly in the fully discrete scheme \eqref{dsequ11}, while the scalar coefficient is introduced only to handle the nonlinear coupling terms. 
%This choice leads to a simpler formulation and analysis while preserving linearity, full decoupling, unconditional energy stability, and optimal-order convergence.
\end{rem}

\begin{rem}
We introduce a single scalar auxiliary variable, together with a novel update equation, to reformulate all nonlinear terms arising from the Cahn--Hilliard and Navier--Stokes equations. The proposed approach is combined with an IMEX time discretization, where the nonlinear terms are treated explicitly. This strategy shares a similar philosophy with the conventional explicit Euler method, where nonlinear terms are evaluated explicitly without the aid of an auxiliary variable. However, the conventional explicit Euler method generally fails to preserve an unconditional energy dissipation law. The proposed SAV reformulation overcomes this limitation by incorporating the nonlinear energy contributions into the auxiliary variable, thereby recovering unconditional energy stability. Furthermore, the auxiliary variable is updated explicitly, allowing the resulting scheme to remain fully decoupled while preserving the energy dissipation property. Consequently, the proposed approach achieves both unconditional energy stability and enhanced computational efficiency.
\end{rem}

\subsection{Stability analysis of the fully discrete scheme}
We assume that 
\begin{eqnarray*}
	 \phi^0\in H^2(\Omega), \ \textbf{u}^0\in \textbf{H}^2(\Omega)\cap \textbf{V}.
\end{eqnarray*}
Under these assumptions, we establish several stability estimates for the numerical solution of scheme \eqref{dsequ11} in various norms throughout this subsection.

\begin{lem}\cite{YY2024} \label{lem5.1}
	For all $m\geq0$, there exist positive constants $M_1,M_2$ such that
	\begin{eqnarray*}
		|\tilde{E}_h^{m+1}|+\Delta t\sum_{n=0}^m(\|\nabla \textbf{u}_h^{n+1}\|_0^2+\|\nabla \mu_h^{n+1}\|_0^2)
		\leq M_1,&\\
		\|\phi_h^{m+1}\|_1^2+\Delta t\sum_{n=0}^m\|\mu_h^{n+1}\|_1^2
		\leq
		M_2,&
	\end{eqnarray*}
	where \[\begin{aligned}M_1&=|\tilde{E}_h^0|,\ \  M_2=M_2^{'}+2M_1(1+\Delta t)+C\Delta t((M_2^{'}+M_1)^3+M_2^{'}), \\
		M_2^{'}&=\left( \|\phi_h^0\|_0^2+M_1\left( 1+4\Delta t+\frac{CM_1^2}{C_0}\right) \right) e^{\frac{CM_1^2}{C_0}}.\end{aligned}\]
\end{lem}

\begin{lem}\cite{YY2024} \label{lem5.2}
	For all $m\geq0$, there exists a positive constant $M_3$ such that
	\begin{eqnarray*}
		\|\triangle_h\phi_h^{m+1}\|_0^2+\Delta t\sum_{n=0}^m\|\triangle_h\mu_h^{n+1}\|_0^2
		\leq
		M_3,
	\end{eqnarray*}
	where $M_3=\left( \|\triangle_h\phi_h^0\|_0^2+C(M_2^3+M_2)+\frac{CM_1^3}{C_0}\right) e^{\frac{CM_1^2}{C_0}}$.
\end{lem}

\begin{lem}\label{lem5.3}
	For all $m\geq0$, there exists a positive constant $M_4$ such that
	\begin{eqnarray*}
		\|\nabla \textbf{u}_h^{m+1}\|_0^2
		+\Delta t\sum_{n=0}^m\|A_h\textbf{u}_h^{n+1}\|_0^2
		\leq
		M_4,
	\end{eqnarray*}
	where  $M_4=\left( \|\nabla\textbf{u}_h^0\|_0^2+\Delta t\|A_h\textbf{u}_h^0\|_0^2+\frac{CM_1M_2}{C_0}(M_1+M_3)\right) e^{\frac{CM_1^4}{C_0^2}}$.
\end{lem}
\begin{proof}
	Taking inner product with $2\Delta tA\textbf{u}_h^{n+1}$ in (\ref{dsequ11}c), we have 
	\begin{eqnarray}\label{A1}
		&&\|\nabla\textbf{u}_h^{n+1}\|_0^2-\|\nabla\textbf{u}_h^n\|_0^2
		+\|\nabla(\textbf{u}_h^{n+1}-\textbf{u}_h^n)\|_0^2
		+2\Delta t\|A\textbf{u}_h^{n+1}\|_0^2\nonumber\\
		&=&2\Delta t\frac{r_h^n}{\sqrt{E_{1,h}^n}}((\mu_h^n \nabla\phi_h^n,A\textbf{u}_h^{n+1})-(\textbf{u}_h^n\cdot\nabla\textbf{u}_h^n,A\textbf{u}_h^{n+1})).
	\end{eqnarray}
	
	% For the right-hand side terms of \eqref{A1}, thanks to \eqref{s1}-\eqref{s3}, the Young inequality and the Cauchy-Schwarz inequality, we have
    For the terms on the right-hand side of \eqref{A1}, by \eqref{s1}--\eqref{s3}, the Cauchy--Schwarz inequality, and Young's inequality, we obtain
	\begin{eqnarray*}
		\left| 2\Delta t\frac{r_h^n}{\sqrt{E_{1,h}^n}}(\mu_h^n\cdot\nabla\phi_h^n,A\textbf{u}_h^{n+1})\right| 
		&\leq&2\Delta t\frac{\left| r_h^n\right| }{\sqrt{E_{1,h}^n}}\|\mu_h^n\|_{L^4}\|\nabla\phi_h^n\|_{L^4}\|A\textbf{u}_h^{n+1}\|_0\\
		&\leq&\frac{\Delta t}{2}\|A\textbf{u}_h^{n+1}\|_0^2+\Delta t\frac{CM_1}{C_0}\|\mu_h^n\|_1^2(M_1+M_3),\\
		\left| 2\Delta t\frac{r_h^n}{\sqrt{E_{1,h}^n}}(\textbf{u}_h^n\cdot\nabla\textbf{u}_h^n,A\textbf{u}_h^{n+1})\right| 
		&\leq& 2\Delta t\frac{\left| r_h^n\right| }{\sqrt{E_{1,h}^n}}\|\textbf{u}_h^n\|_{L^\infty}\|\nabla\textbf{u}_h^n\|_0\|A\textbf{u}_h^{n+1}\|_0\\
		&\leq& C\Delta t\frac{\left| r_h^n\right| }{\sqrt{E_{1,h}^n}}\|\textbf{u}_h^n\|_0^{1/2}\|\nabla\textbf{u}_h^n\|_0\|A\textbf{u}_h^n\|_0^{1/2}\|A\textbf{u}_h^{n+1}\|_0\\
		&\leq&\frac{\Delta t}{2}\|A\textbf{u}_h^{n+1}\|_0^2+\Delta t\frac{CM_1}{C_0}\|\textbf{u}_h^n\|_0\|\nabla\textbf{u}_h^n\|_0^2\|A\textbf{u}_h^n\|_0\\
		&\leq&\frac{\Delta t}{2}\|A\textbf{u}_h^{n+1}\|_0^2+\frac{\Delta t}{2}\|A\textbf{u}_h^n\|_0^2+\Delta t\frac{CM_1^2}{C_0^2}\|\textbf{u}_h^n\|_0^2\|\nabla\textbf{u}_h^n\|_0^4.
	\end{eqnarray*}
Combining the above inequalities with \eqref{A1} and summing from $ n =0 $ to $ m $, we obtain
	\begin{eqnarray*}
		&&\|\nabla \textbf{u}_h^{m+1}\|_0^2+\frac{\Delta t}{2}\|A\textbf{u}_h^{m+1}\|_0^2
		+\frac{\Delta t}{2}\sum_{n=0}^m\|A\textbf{u}_h^{n+1}\|_0^2
		\\
		&\leq& \|\nabla\textbf{u}_h^0\|_0^2+\frac{\Delta t}{2}\|A\textbf{u}_h^0\|_0^2+\frac{CM_1^2}{C_0^2}\Delta t\sum_{n=0}^m\|\textbf{u}_h^n\|_0^2\|\nabla\textbf{u}_h^n\|_0^4\\
		&&+\frac{CM_1}{C_0}\Delta t\sum_{n=0}^m\|\mu_h^n\|_1^2(M_1+M_3),
	\end{eqnarray*}
	% which together with the Gronwall's Lemma with
	% $ \sigma_n = 1 $ for all $ n \geq 0 $, Lemma \ref{lem1} and Lemma \ref{lem2}, lead to the desired result.
    which, together with the discrete Gronwall's Lemma (with $\sigma_n=1$ for all $n\ge 0$) and Lemmas \ref{lem5.1} and \ref{lem5.2}, yields the desired result.
\end{proof}

%	\hskip\parindent
%	\section{Convergence analysis of Euler semi-implicit-SAV scheme}\label{secSemiEst}
%	\setcounter{equation}{0}
\section{Error estimates for fully discrete scheme} \label{secError}
This section is devoted to establishing the error estimates for scheme \eqref{dsequ11}. First, we derive error estimates for the phase field, velocity, and auxiliary variable based on the stability results established above. Then, an error estimate for the pressure is provided. 

By Taylor's formula with integral remainder, we have
\[
f(t^{n+1})=f(t^n)+\Delta tf_t(t^n)+\int_{t^n}^{t^{n+1}}(t^{n+1}-t)f_{tt}(t)dt,
\]
and 
\[
f(t^n)=f(t^{n+1})-\Delta tf_t(t^{n+1})+\int_{t^{n+1}}^{t^n}(t^n-t)f_{tt}(t)dt.
\]
Then we discretize equations \eqref{1sequ} at the time level as
	\begin{subequations}\label{sequ}
		\begin{align}
			(d_t\phi(t^{n+1}),\psi)&+(\nabla \mu(t^{n+1}),\psi)
			+\frac{r(t^{n+1})}{\sqrt{E_1(t^{n+1})}}(\textbf{u}(t^{n+1})\cdot\nabla\phi(t^{n+1}),\psi)\nonumber\\
			&- (R_\phi,\psi) =0,\label{sequ1}\\
			(\mu(t^{n+1}),\tau)&-(\nabla\phi(t^{n+1}),\nabla\tau)
			-(g(\phi(t^{n+1})),\tau)=0,\label{sequ2}\\
			(d_t\textbf{u}(t^{n+1}),\textbf{v})&+\frac{r(t^{n+1})}{\sqrt{E_1(t^{n+1})}}(\textbf{u}(t^{n+1})\cdot\nabla \textbf{u}(t^{n+1}),\textbf{v})+(\nabla p(t^{n+1}),\textbf{v})+(\nabla \textbf{u}(t^{n+1}),\nabla\textbf{v}) \nonumber\\
			&-\frac{r(t^{n+1})}{\sqrt{E_1(t^{n+1})}}(\mu(t^{n+1})\nabla\phi(t^{n+1}),\textbf{v})- (R_u,\textbf{v})
			=0,\label{sequ3}\\
			(\nabla\cdot \textbf{u}(t^{n+1}),q)&=0,\label{sequ5}\\
			d_tr(t^{n+1})&=\frac{1}{2r(t^n)} (g(\phi(t^n)),d_t\phi(t^{n+1}))+\frac{1}{2\sqrt{E_1(t^n)}}((\textbf{u}(t^n)\cdot\nabla \textbf{u}(t^n),\textbf{u}(t^n))
			\nonumber\\
			&+(\textbf{u}(t^n)\nabla \phi(t^n),\mu(t^n))-(\mu(t^n)\nabla \phi(t^n),\textbf{u}(t^n)))-T_1 +R_r,\label{sequ6}
		\end{align}
	\end{subequations}
	where
	\begin{eqnarray*}
		&&R_\phi=\frac{1}{\Delta t}\int_{t^n}^{t^{n+1}}(t^n-t)\phi_{tt}(t)dt,\\
		&&R_u=\frac{1}{\Delta t}\int_{t^n}^{t^{n+1}}(t^n-t)\textbf{u}_{tt}(t)dt,\\
		&&R_r=\frac{1}{\Delta t}\int_{t^n}^{t^{n+1}}(t^{n+1}-t)r_{tt}(t)dt,\\ 
		&&T_1=\frac{1}{2\Delta tr(t^n)}\left( g(\phi(t^n)),\int_{t^n}^{t^{n+1}}(t^{n+1}-t)\phi_{tt}(t)dt\right).
	\end{eqnarray*}

We impose the following regularity assumption on the solution.

\begin{assumption}\label{aassume}
Assume that the solution $(\phi,\mu,\mathbf{u})$ of system \eqref{chnsequ} satisfies
        \begin{eqnarray}
		\begin{cases}
			\textbf{u}\in L^\infty(0,T;\textbf{H}^2(\Omega)),\ \ \phi,\mu\in L^\infty(0,T;H^2(\Omega)),\\
				\phi_t\in  L^\infty(0,T;L^2(\Omega))\cap L^2(0,T;H^1(\Omega)),\\
			\textbf{u}_t\in L^\infty(0,T;\textbf{L}^2(\Omega)),\ \ 
			\mu_t\in L^\infty(0,T;H^1(\Omega)),\\
				\textbf{u}_{tt}\in L^2(0,T;\textbf{H}^{-1}(\Omega)),\ \ \phi_{tt}\in L^2(0,T;L^2(\Omega)).
		\end{cases}
	\end{eqnarray}
\end{assumption}
    
	% We impose the following assumption.
 %    \begin{assumption}
 %    We assume the solutions $ (\phi, \mu, \textbf{u}) $ of the system \eqref{chnsequ} possess the following regularity assumption
 %        \begin{eqnarray}\label{aassume}
	% 	\begin{cases}
	% 		\textbf{u}\in L^\infty(0,T;\textbf{H}^2(\Omega)),\ \ \phi,\mu\in L^\infty(0,T;H^2(\Omega)),\\
	% 			\phi_t\in  L^\infty(0,T;L^2(\Omega))\cap L^2(0,T;H^1(\Omega)),\\
	% 		\textbf{u}_t\in L^\infty(0,T;\textbf{L}^2(\Omega)),\ \ 
	% 		\mu_t\in L^\infty(0,T;H^1(\Omega)),\\
	% 			\textbf{u}_{tt}\in L^2(0,T;\textbf{H}^{-1}(\Omega)),\ \ \phi_{tt}\in L^2(0,T;L^2(\Omega)).
	% 	\end{cases}
	% \end{eqnarray}
 %    \end{assumption}
	
	The following truncation error estimates can be established, provided that the exact solutions satisfy \Cref{aassume}.
    
	\begin{lem}\cite{YY2024} \label{truncation}
		Under \Cref{aassume}, for all $  m\geq 0 $, the truncation errors 
		%$R_\phi, R_u, R_r, T_1   $ 
		satisfy
		\begin{eqnarray*}
\sum_{n=0}^m(\|R_\phi\|_{-1}^2+\|R_u\|_{-1}^2+\|T_1\|_0^2+|R_r|^2)
			\leq
			C\Delta t,
		\end{eqnarray*}
	\end{lem}

\begin{lem}\label{lem2.1}
Let $E_1(t^n) = \int_\Omega F(\phi(t^n)) \,dx$ and $E_{1,h}^n = \int_\Omega F(\phi_h^n) \,dx$ denote the continuous and discrete energies at time $t^n$, respectively. Assume that $F(\phi)$ is Lipschitz continuous on the range of $\phi$. Then there exists a constant $C > 0$, independent of the mesh size $h$ and time step $\Delta t$, such that
\begin{eqnarray*}
\left|
\frac{1}{\sqrt{E_1(t^n)}} - \frac{1}{\sqrt{E_{1,h}^n}}
\right|
&\le& C \|e_{\phi h}^n\|_0,
\end{eqnarray*}
where $e_{\phi h}^n = \phi(t^n) - \phi_h^n$ is the phase-field error.
\end{lem}

\begin{proof}
Using the algebraic identity
\[
\frac{1}{\sqrt{A}} - \frac{1}{\sqrt{B}}
= \frac{B - A}{\sqrt{AB}(\sqrt{A} + \sqrt{B})},
\]
we obtain
\[
\left|
\frac{1}{\sqrt{E_1(t^n)}} - \frac{1}{\sqrt{E_{1,h}^n}}
\right|
=
\frac{|E_1(t^n) - E_{1,h}^n|}
{\sqrt{E_1(t^n)E_{1,h}^n}
\left(\sqrt{E_1(t^n)} + \sqrt{E_{1,h}^n}\right)}.
\]
By the positivity assumption $E_1(t^n), E_{1,h}^n \ge C_0' > 0$, the denominator satisfies
\[
\sqrt{E_1(t^n)E_{1,h}^n}
\left(\sqrt{E_1(t^n)} + \sqrt{E_{1,h}^n}\right)
\ge C_0' \cdot 2\sqrt{C_0'} =: C_1 > 0,
\]
so that
\[
\left|
\frac{1}{\sqrt{E_1(t^n)}} - \frac{1}{\sqrt{E_{1,h}^n}}
\right|
\le \frac{1}{C_1} |E_1(t^n) - E_{1,h}^n|.
\]
Since $F$ is Lipschitz continuous on the range of $\phi$, there exists $L>0$ such that $|F(a)-F(b)| \le L|a-b|$. Therefore,
\[
|E_1(t^n) - E_{1,h}^n|
= \left|\int_\Omega \big(F(\phi(t^n)) - F(\phi_h^n)\big)\,dx\right|
\le L \int_\Omega |\phi(t^n) - \phi_h^n|\,dx
\le L |\Omega|^{1/2} \|e_{\phi h}^n\|_0.
\]
Combining the above estimates yields the desired result with $C = L|\Omega|^{1/2}/C_1$.
\end{proof}

\begin{lem}\cite{YY2024}\label{final}
	For all $m\geq0$, there exists a positive constant $C$ such that
	\begin{eqnarray*}
		\|d_t\phi_h^{m+1}\|_1
		\leq C.
	\end{eqnarray*}
\end{lem}

Define the errors
\begin{align*}
	e_{\phi h}^n&=\phi(t^n)-\phi_h^n,\ \ \Phi_{\phi h}^n=R_h\phi(t^n)-\phi_h^n,\ \
	\Theta_{\phi h}^n=\phi(t^n)-R_h\phi(t^n),\\
	e_{\mu h}^n&=\mu(t^n)-\mu_h^n,\ \ \Phi_{\mu h}^n=R_h\mu(t^n)-\mu_h^n,\ \
	\Theta_{\mu h}^n=\mu(t^n)-R_h\mu(t^n),\\
	e_{u h}^n&=\textbf{u}(t^n)-\textbf{u}_h^n,\ \ \Phi_{u h}^n=P_h\textbf{u}(t^n)-\textbf{u}_h^n,\ \
	\Theta_{u h}^n=\textbf{u}(t^n)-P_h\textbf{u}(t^n),\\
	e_{r h}^n&=r(t^n)-r_h^n,\ \ \ \ e_{p h}^n=p(t^n)-p_h^n.
\end{align*}

By subtracting \eqref{dsequ11} from \eqref{sequ}, we obtain the error equations
\begin{subequations}\label{eequ}
	\begin{align}
		(d_te_{\phi h}^{n+1},\psi_h)&+(\nabla e_{\mu h}^{n+1},\nabla\psi_h)+\left( \frac{r(t^{n+1})}
		{\sqrt{E_1(t^{n+1})}}\textbf{u}(t^{n+1})\cdot\nabla\phi(t^{n+1})
		-\frac{r_h^n}
		{\sqrt{E_{1, h}^n}}\textbf{u}_h^n\cdot\nabla\phi_h^n,\psi_h\right) \nonumber\\
		&\ \ \ +(R_\phi,\psi_h)=0,\\
		(e_{\mu h}^{n+1},\tau_h)&-(\nabla e_{\phi h}^{n+1},\nabla\tau_h)-
		(g(\phi(t^{n+1}))-g(\phi_h^n),\tau_h)=0,\\
		(d_te_{u h}^{n+1},\textbf{v}_h)&+(\nabla e_{u h}^{n+1},\nabla \textbf{v}_h) -(e_{p h}^{n+1},\nabla \cdot \textbf{v}_h)-(R_u,\textbf{v}_h)\nonumber\\
		&\ \ \ +\left( \frac{r(t^{n+1})}
		{\sqrt{E_1(t^{n+1})}}\textbf{u}(t^{n+1})\cdot\nabla \textbf{u}(t^{n+1})-\frac{r_h^n}
		{\sqrt{E_{1, h}^n}}\textbf{u}_h^n\cdot\nabla \textbf{u}_h^n,\textbf{v}_h\right)\nonumber\\
		&\ \ \ -\left( \frac{r(t^{n+1})}
		{\sqrt{E_1(t^{n+1})}}\mu(t^{n+1})\nabla\phi(t^{n+1})-\frac{r_h^n}
		{\sqrt{E_{1, h}^n}}\mu_h^n\nabla\phi_h^n,\textbf{v}_h\right) =0,\\
		(\nabla\cdot e_{u h}^{n+1},q_h)&=0,\\
		d_te_{rh}^{n+1}&=\frac{1}{2r(t^n)}(g(\phi(t^n)),d_t\phi(t^{n+1}))-\frac{1}{2r_h^{n+1}}(g(\phi_h^n),d_t\phi_h^{n+1})\nonumber\\
		&\ \ \ +\frac{1}{2\sqrt{E_1(t^n)}}((\textbf{u}(t^n)\cdot\nabla \textbf{u}(t^n),\textbf{u}(t^n)) 
		+(\textbf{u}(t^n)\cdot\nabla\phi(t^n),\mu(t^n))\nonumber\\
		&\ \ \ \ \ \ \ \ \ \ \ \ \ \ \ \ \ \ \ \ \ -(\mu(t^n)\nabla\phi(t^n),\textbf{u}(t^n)))\nonumber\\
		&\ \ \ -\frac{r_h^n}{2r_h^{n+1}\sqrt{E_{1,h}^n}}((\textbf{u}_h^n\cdot\nabla \textbf{u}_h^n,\textbf{u}_h^{n+1})
		+(\textbf{u}_h^n\cdot\nabla\phi_h^n,\mu_h^{n+1})-(\mu_h^n\nabla\phi_h^n,\textbf{u}_h^{n+1}))\nonumber\\
		&\ \ \ -T_1+R_r.
	\end{align}
\end{subequations}

% By analogy with the proofs of Lemma 3.2 and Lemma 3.4 in Reference \cite{YY2024}, with $l = 1$, we can obtain the following two conclusions.
Following the proofs of Lemmas 3.2 and 3.4 in \cite{YY2024}, with $l=1$, we can derive the following two conclusions.
\begin{lem}\label{thm4.1}
	Under \Cref{aassume}, for all $  m\geq 0 $ and $e_{\phi h}^0=0$, it holds
	\begin{eqnarray*}
		\|\Phi_{\phi h}^{m+1}\|_1^2+\Delta t
	\sum_{n=0}^m\| e_{\mu h}^{n+1}\|_1^2
	&\leq&
	C \Big{(} \Delta t^2+h^{2}+ \Delta t\sum_{n=0}^m(\|\phi(t^n)\|_2^2 +\|\nabla\textbf{u}_h^n\|_0^2\\
	&&+\|\nabla\textbf{u}_h^n\|_0^2\|\nabla\phi_h^n\|_0^2)(\|\Phi_{uh}^{n+1}\|_0^2+|e_{rh}^{n+1}|^2)\Big{)} .
	\end{eqnarray*}
\end{lem}
\begin{lem}\label{r}
	Under \Cref{aassume},
	for all $m\geq0$ and $e_{rh}^0=0$, it holds
	\begin{eqnarray*}
		|e_{rh}^{m+1}|^2
		&\leq&
		C\Big{(}\Delta t^2+h^{2}+\frac{\Delta t}{4}\sum_{n=0}^m\|e_{\mu h}^{n+1}\|_1^2+\frac{\Delta t}{2}\sum_{n=0}^m\|\nabla\Phi_{uh}^{n+1}\|_0^2\\
		&&+C \Delta t\sum_{n=0}^m(\|\phi(t^n)\|_2^2 +\|\nabla\textbf{u}_h^n\|_0^2(\|\phi_h^n\|_1^2+\|\textbf{u}_h^n\|_0^2)+\|\mu_h^{n+1}\|_1^2\\
		&&+(\|\phi(t^n)\|_2^2+\|\nabla\textbf{u}_h^n\|_0^2)\|\mu(t^n)\|_1^2+\|\nabla\phi(t^n)\|_0^2\|\nabla\textbf{u}(t^n)\|_0^2\\
		&&+\|\nabla\phi_h^n\|_0^2\|\mu_h^n\|_1^2+\|\nabla\textbf{u}(t^n)\|_0^2\|\mu_h^n\|_1^2+\|\nabla\textbf{u}_h^{n+1}\|_0^2\\
		&&+\|\phi_t\|^2_{L^\infty(0,T;H^1(\Omega))})(
		\|\Phi_{\phi h}^{n+1}\|_1^2+\|\Phi_{uh}^{n+1}\|_0^2)\Big{)}.
	\end{eqnarray*}
\end{lem}

Next, we prove the error estimation for the velocity $\textbf{u}$.
\begin{lem}\label{lemu}
	Under \Cref{aassume}, for all $  m\geq 0 $ and $e_{u h}^0=0$, it holds
	\begin{eqnarray*}
		\|\Phi_{u h}^{m+1}\|_0^2+\Delta t
	\sum_{n=0}^m\|\nabla \Phi_{u h}^{n+1}\|_0^2
	&\leq&
	C \Big{(} \Delta t^2+h^{2}+\frac{\Delta t}{4}\sum_{n=0}^m\|e_{\mu h}^{n+1}\|_1^2\\
    &&+ \Delta t\sum_{n=0}^m(\|\phi(t^n)\|_2^2+\|\mu_h^n\|_1^2\|\nabla\phi_h^n\|_0^2+\|\nabla\textbf{u}(t^n)\|_0^2\\
    &&+\|\textbf{u}_h^n\|_0^2\|\nabla\textbf{u}_h^n\|_0^2)(\|\Phi_{\phi h}^{n+1}\|_0^2+|e_{rh}^{n+1}|^2)\Big{)} .
	\end{eqnarray*}
\end{lem}
\begin{proof}
Taking $\textbf{v}_h=2\Delta t\Phi_{u h}^{n+1}$ in (\ref{eequ}c) and using the projection \eqref{orth1}, we obtain
\begin{eqnarray}\label{2error}
	&&\|\Phi_{uh}^{n+1}\|_0^2-\|\Phi_{uh}^n\|_0^2+\|\Phi_{uh}^{n+1}-\Phi_{uh}^n\|_0^2+\Delta t(\|\nabla\Phi_{u h}^{n+1}\|_0^2+\|\nabla e_{u h}^{n+1}\|_0^2)
	\nonumber\\
	&=&\Delta t\|\nabla\Theta_{u h}^{n+1}\|_0^2+2\Delta t(R_u,\Phi_{u h}^{n+1})\nonumber\\
	&&
	+2\Delta t\left( \left( \frac{r(t^{n+1})}
	{\sqrt{E_1(t^{n+1})}}\mu(t^{n+1})\nabla\phi(t^{n+1})-\frac{r_h^n}
	{\sqrt{E_{1, h}^n}}\mu_h^n\nabla\phi_h^n,\Phi_{u h}^{n+1}\right) \right.\nonumber\\
	&&-\left. \left(\frac{r(t^{n+1})}{\sqrt{E_1(t^{n+1})}}\textbf{u}(t^{n+1})\cdot\nabla \textbf{u}(t^{n+1})-\frac{r_h^n}
	{\sqrt{E_{1, h}^n}}\textbf{u}_h^n\cdot\nabla \textbf{u}_h^n,\Phi_{u h}^{n+1}\right)  \right) \nonumber\\
	&=&\Delta t\|\nabla\Theta_{u h}^{n+1}\|_0^2+\sum_{i=1}^{3}H_i.
\end{eqnarray}	

We now estimate the terms on the right-hand side of \eqref{2error} one by one. For $H_1$ and $H_2$, it follows
\begin{eqnarray*}
	|H_1|&=&|2\Delta t(R_u,\Phi_{uh}^{n+1})|\leq\frac{\Delta t}{3}\|\nabla\Phi_{uh}^{n+1}\|_0^2+C\Delta t\|R_u\|_{-1}^2,
\end{eqnarray*}
\begin{eqnarray*}
	|H_2|&=&\left| 2\Delta t\left( \frac{r(t^{n+1})}
	{\sqrt{E_1(t^{n+1})}}\mu(t^{n+1})\nabla\phi(t^{n+1})-\frac{r_h^n}
	{\sqrt{E_{1, h}^n}}\mu_h^n\nabla\phi_h^n,\Phi_{u h}^{n+1}\right)\right| \\
	&=&2\Delta t \Bigg{|} (\mu(t^{n+1})\nabla\phi(t^{n+1}),\Phi_{u h}^{n+1})-(\mu(t^n)\nabla\phi(t^n),\Phi_{u h}^{n+1})\\
	&&+(\mu(t^n)\nabla\phi(t^n),\Phi_{u h}^{n+1})
	-(\mu_h^n\nabla\phi_h^n,\Phi_u^{n+1})
	+\frac{e_{rh}^n}{\sqrt{E_1(t^n)}}(
	\mu_h^n\nabla\phi_h^n,\Phi_{u h}^{n+1})\\
	&&\left.
	+r_h^n\left( \frac{1}{\sqrt{E_1(t^n)}}-\frac{1}{\sqrt{E_{1,h}^n}}\right) (
	\mu_h^n\nabla\phi_h^n,\Phi_{u h}^{n+1})\right|.
\end{eqnarray*}
%Thanks to \eqref{s1}, \eqref{s3} and the Young inequality, we have
By \eqref{s1}, \eqref{s3}, and Young's inequality, we have
\begin{eqnarray*}
	&&2\Delta t|(\mu(t^{n+1})\nabla\phi(t^{n+1}),\Phi_{u h}^{n+1})-(\mu(t^n)\nabla\phi(t^n),\Phi_{u h}^{n+1})| \\
	&=&2\Delta t|((\mu(t^{n+1})-\mu(t^n))\nabla\phi(t^{n+1})+\mu(t^n)\nabla(\phi(t^{n+1})-\phi(t^n)),\Phi_{u h}^{n+1})|\\
	&\leq&\frac{\Delta t}{12}\|\nabla
	\Phi_{u h}^{n+1}\|_0^2+C\Delta t^3(\|\mu_t\|_{L^\infty(0,T;H^1(\Omega))}^2\|\nabla\phi(t^{n+1})\|_0^2+\|\mu(t^n)\|_1^2\|\phi_t\|_{L^\infty(0,T;H^1(\Omega))}^2),\\
	&&2\Delta t|(\mu(t^n)\nabla\phi(t^n),\Phi_{uh}^{n+1})
	-(\mu_h^n\nabla\phi_h^n,\Phi_{uh}^{n+1})| \\
	&=&2\Delta t|(e_{\mu h}^n\nabla\phi(t^n),\Phi_{uh}^{n+1})
	+(\mu_h^n\nabla e_{\phi h}^n,\Phi_{uh}^{n+1})|\\
	&\leq&\frac{\Delta t}{12}\|\nabla \Phi_{u h}^{n+1}\|_0^2+\frac{\Delta t}{8}\| e_{\mu h}^n\|_1^2
	+C\Delta t(\|\mu_h^n\|_1^2+\|\phi(t^n)\|_2^2)(\|\nabla \Phi_{\phi h}^n\|_0^2\\
    &&+\|\nabla \Theta_{\phi h}^n\|_0^2+\|\Phi_{uh}^{n+1}\|_0^2),\\
	&&2\Delta t\left| \frac{e_{rh}^n}{\sqrt{E_1(t^n)}}(
	\mu_h^n\nabla\phi_h^n,\Phi_{u h}^{n+1})\right| \\
    &\leq& 2\Delta t\frac{|e_{rh}^n|}{\sqrt{C_0}}\|\mu_h^n\|_1\|\nabla\phi_h^n\|_0\|\nabla
	\Phi_{uh}^{n+1}\|_0\\
	&\leq&\frac{\Delta t}{12}\|\nabla
	\Phi_{uh}^{n+1}\|_0^2+C\Delta t\|\mu_h^n\|_1^2\|\nabla\phi_h^n\|_0^2|e_{rh}^n|^2,
\end{eqnarray*}
and using Lemma \ref{lem2.1}, it follows
\begin{eqnarray*}	
	&&2\Delta t\left| r_h^n\left( \frac{1}{\sqrt{E_1(t^n)}}-\frac{1}{\sqrt{E_{1,h}^n}}\right) (
	\mu_h^n\nabla\phi_h^n,\Phi_{u h}^{n+1})\right|\\
	&\leq&C\Delta t\|e_{\phi h}^n\|_0\|\mu_h^n\|_1\|\nabla\phi_h^n\|_0\|\nabla
	\Phi_{u h}^{n+1}\|_0\\
	&\leq&\frac{\Delta t}{12}\|\nabla
	\Phi_{u h}^{n+1}\|_0^2+C\Delta t\|\mu_h^n\|_1^2\|\nabla\phi_h^n\|_0^2(\| \Phi_{\phi h}^n\|_0^2+\|\Theta_{\phi h}^n\|_0^2).
\end{eqnarray*}
Thus, it holds
\begin{eqnarray*}
	|H_2|&\leq& \frac{\Delta t}{3}\|\nabla
	\Phi_{u h}^{n+1}\|_0^2+C\Delta t^3(\|\mu_t\|_{L^\infty(0,T;H^1(\Omega))}^2\|\nabla\phi(t^{n+1})\|_0^2+\|\mu(t^n)\|_1^2\|\phi_t\|_{L^\infty(0,T;H^1(\Omega))}^2)\\
    &&+\frac{\Delta t}{4}\| e_{\mu h}^n\|_1^2
	+C\Delta t(\|\mu_h^n\|_1^2+\|\phi(t^n)\|_2^2)(\|\nabla \Phi_{\phi h}^n\|_0^2+\|\nabla \Theta_{\phi h}^n\|_0^2+\|\Phi_{uh}^{n+1}\|_0^2)\\
    &&+C\Delta t\|\mu_h^n\|_1^2\|\nabla\phi_h^n\|_0^2|e_{rh}^n|^2+C\Delta t\|\mu_h^n\|_1^2\|\nabla\phi_h^n\|_0^2(\| \Phi_{\phi h}^n\|_0^2+\|\Theta_{\phi h}^n\|_0^2)
\end{eqnarray*}

$H_3$ can be derived similarly,
\begin{eqnarray*}
	|H_3|&=&\left| 2\Delta t\left(\frac{r(t^{n+1})}{\sqrt{E_1(t^{n+1})}}\textbf{u}(t^{n+1})\cdot\nabla \textbf{u}(t^{n+1})-\frac{r_h^n}
	{\sqrt{E_{1, h}^n}}\textbf{u}_h^n\cdot\nabla \textbf{u}_h^n,\Phi_{u h}^{n+1}\right)\right| \\
	&=&2\Delta t\Bigg{|}(\textbf{u}(t^{n+1})\cdot\nabla\textbf{u}(t^{n+1}),\Phi_{u h}^{n+1})-(\textbf{u}(t^n)\cdot\nabla\textbf{u}(t^n),\Phi_{u h}^{n+1})\\
	&&\left.+(\textbf{u}(t^n)\cdot\nabla\textbf{u}(t^n),\Phi_{u h}^{n+1})
	-(\textbf{u}_h^n\cdot\nabla\textbf{u}_h^n,\Phi_{u h}^{n+1})
	+\frac{e_{rh}^n}{\sqrt{E_1(t^n)}}(
	\textbf{u}_h^n\cdot\nabla\textbf{u}_h^n,\Phi_{u h}^{n+1})\right.\\
	&&\left.+r_h^n\left( \frac{1}{\sqrt{E_1(t^n)}}-\frac{1}{\sqrt{E_{1,h}^n}}\right) (
	\textbf{u}_h^n\cdot\nabla\textbf{u}_h^n,\Phi_{u h}^{n+1})\right|.\\
&\leq&\frac{\Delta t}{3}\|\nabla
	\Phi_{u h}^{n+1}\|_0^2+C\Delta t^3\|\textbf{u}_t\|_{L^\infty(0,T;L^2(\Omega))}^2(\|A\textbf{u}(t^{n+1})\|_0^2+\|A\textbf{u}(t^n)\|_0^2)\\
    &&+C\Delta t(\|\nabla\textbf{u}(t^n)\|_0^2+\|\nabla\textbf{u}_h^n\|_0^2)(\|\Phi_{uh}^n\|_0^2+\|\Theta_{uh}^n\|_0^2)\\
    &&+C\Delta t\|\textbf{u}_h^n\|_0^2\|\nabla\textbf{u}_h^n\|_0^2|e_{rh}^n|^2
    +C\Delta t\|\textbf{u}_h^n\|_0^2\|\nabla\textbf{u}_h^n\|_0^2(\| \Phi_{\phi h}^n\|_0^2+\|\Theta_{\phi h}^n\|_0^2).
\end{eqnarray*}

Substituting the above inequalities into \eqref{2error}, summing from $ n = 0 $ to $ m $, we get	
		\begin{eqnarray*}
	&&\|\Phi_{uh}^{m+1}\|_0^2+\Delta t\sum_{n=0}^m\|\nabla \Phi_{u h}^{n+1}\|_0^2\\
	&\leq&\|\Phi_{uh}^0\|_0^2+C \Delta t\sum_{n=0}^m\|R_u\|_{-1}^2+\frac{\Delta t}{4}\sum_{n=0}^m\| e_{\mu h}^{n+1}\|_1^2\\
	&&+C \Delta t^3\sum_{n=0}^m(\|A\textbf{u}(t^{n+1})\|_0^2+\|\nabla\phi(t^{n+1})\|_0^2+\|\mu(t^n)\|_1^2)(\|\textbf{u}_t\|^2_{L^\infty(0,T;L^2
		(\Omega))}\\
	&&\ \ \ \ \ \ \ \ \ \ \ \ \ \ \ \  +\|\phi_t\|^2_{L^\infty(0,T;H^1(\Omega))}+\|\mu_t\|^2_{L^\infty(0,T;H^1(\Omega))})\\
	&&+C \Delta t\sum_{n=0}^m(\|\phi(t^n)\|_2^2+\|\mu_h^n\|_1^2\|\nabla\phi_h^n\|_0^2+\|\nabla\textbf{u}(t^n)\|_0^2+\|\textbf{u}_h^n\|_0^2\|\nabla\textbf{u}_h^n\|_0^2)(\|\Theta_{\phi h}^{n+1}\|_1^2\\
	&&\ \ \ \ \ \ \ \ \ \ \ \ \ \ \ \  
	+\|\Theta_{u h}^{n+1}\|_1^2+\|\Phi_{\phi h}^{n+1}\|_1^2+\|\Phi_{uh}^{n+1}\|_0^2+|e_{rh}^{n+1}|^2).
\end{eqnarray*}
Using  \eqref{orth1}, \eqref{orth3}, Lemma \ref{truncation}, Lemma \ref{thm4.1},  Lemma \ref{r}, \Cref{aassume}, and Gronwall's lemma, we finish the proof.
\end{proof}

\begin{prop}\label{prop}
	Under the conditions of Lemma \ref{thm4.1}-Lemma \ref{lemu}, it holds
	\begin{eqnarray*}
		&&\|\phi(t^{m+1})-\phi_h^{m+1}\|_1^2+|r(t^{m+1})-r_h^{m+1}|^2+\Delta t\sum_{n=0}^m(\|\nabla(\textbf{u}(t^{n+1})-\textbf{u}_h^{n+1})\|_0^2+\|\mu(t^{n+1})-\mu_h^{n+1}\|_1^2)\\
		&\leq& C(\Delta t^2+h^2).
	\end{eqnarray*}
\end{prop}

%Then we prove the optimal $L^2$-norm error estimates of the numerical solutions for fully discrete SAV scheme \eqref{dsequ11}.
We next prove the optimal $L^2$-norm error estimates for the numerical solutions of the fully discrete SAV-FEM  scheme \eqref{dsequ11}.
\begin{thm}\label{thm5.5}
	Under \Cref{aassume},
	for all $m\geq0$, $d_n\Delta t<1 $ and $e_{\phi h}^0=e_{u h}^0=e_{rh}^0=0$, it holds
	\begin{eqnarray*}
		\| e_{\phi h}^{m+1}\|_0^2+\|e_{u h}^{m+1}\|_{-1}^2+(e_{r h}^{m+1})^2
		+
		\sum_{n=0}^m\| e_{\mu h}^{n+1}\|_0^2
		\leq
		C e^{\left( \Delta t\sum\limits_{n=0}^{m}(1-d_n\Delta t)^{-1}d_n\right) }(\Delta t^2+h^4),
	\end{eqnarray*}
	where $d_n>0$ is a bounded constant.
\end{thm}
	\begin{proof}
		Choosing $(\psi_h,\tau_h,\textbf{v}_h)=2\Delta t(\triangle^{-1}_h\Phi_{\mu h}^{n+1},d_t\triangle^{-1}_h\Phi_{\phi h}^{n+1},A_h^{-1}\Phi_{u h}^{n+1})$ in (\ref{eequ}a)-(\ref{eequ}c), and multiplying (\ref{eequ}e) by $2e_{rh}^{n+1}$,  and using \eqref{Ah-1}-\eqref{Ah-11}, we get
		\begin{eqnarray}\label{w}
			&&\| \Phi_{\phi h}^{n+1}\|_0^2-\|\Phi_{\phi h}^n\|_0^2+\|\Phi_{\phi h}^{n+1}-\Phi_{\phi h}^n\|_0^2
			+\|\Phi_{uh}^{n+1}\|_{-1}^2-\|\Phi_{uh}^n\|_{-1}^2+\|\Phi_{uh}^{n+1}-\Phi_{uh}^n\|_{-1}^2\nonumber\\
			&&+|e_{rh}^{n+1}|^2-|e_{rh}^n|^2+|e_{rh}^{n+1}-e_{rh}^n|^2+\Delta t(\|\Phi_{\mu h}^{n+1}\|_0^2+\|e_{\mu h}^{n+1}\|_0^2+\|\Phi_{uh}^{n+1}\|_0^2+\|e_{uh}^{n+1}\|_0^2)\nonumber\\
			&=&\Delta t(\|\Theta_{\mu h}^{n+1}\|_0^2+\|\Theta_{u h}^{n+1}\|_0^2)
			+2\Delta t((R_u,A_h^{-1}\Phi_{u h}^{n+1})+(R_\phi,\triangle_h^{-1}\Phi_{\mu h}^{n+1})-(\Theta_{\phi h}^{n+1},d_t\Phi_{\phi h}^{n+1})\nonumber\\
    &&+e_{rh}^{n+1}(R_r-T_1))+2\Delta t(g(\phi(t^{n+1}))-
			g(\phi_h^n),d_t\triangle^{-1}_h\Phi_{\phi h}^{n+1})\nonumber\\
            &&+\Delta te_{rh}^{n+1}\left( \left(\frac{g(\phi(t^n))}{r(t^n)}-\frac{g(\phi_h^n)}{r_h^{n+1}},d_t\phi_h^{n+1}\right) +\left(  \frac{g(\phi(t^n))}{r(t^n)},d_te_{\phi h}^{n+1}\right) \right)\nonumber\\
			&&
			+2\Delta t\left( \left( \frac{r(t^{n+1})}
			{\sqrt{E_1(t^{n+1})}}\mu(t^{n+1})\nabla\phi(t^{n+1})-\frac{r_h^n}
			{\sqrt{E_{1, h}^n}}\mu_h^n\nabla\phi_h^n,A_h^{-1}\Phi_{u h}^{n+1}\right) \right.\nonumber\\
			&&-\left(\frac{r(t^{n+1})}{\sqrt{E_1(t^{n+1})}}\textbf{u}(t^{n+1})\cdot\nabla \textbf{u}(t^{n+1})-\frac{r_h^n}
			{\sqrt{E_{1, h}^n}}\textbf{u}_h^n\cdot\nabla \textbf{u}_h^n,A_h^{-1}\Phi_{u h}^{n+1}\right) \nonumber\\
			&&\left. +
			\left(\frac{r(t^{n+1})}
			{\sqrt{E_1(t^{n+1})}}\textbf{u}(t^{n+1})\cdot\nabla\phi(t^{n+1})-\frac{r_h^n}
			{\sqrt{E_{1, h}^n}}\textbf{u}_h^n\cdot\nabla\phi_h^n,
			\triangle^{-1}_h\Phi_{\mu h}^{n+1}
			\right) \right) 
			 \nonumber\\
			&&+\Delta te_{rh}^{n+1}\left( \frac{1}{\sqrt{E_1(t^n)}}(\textbf{u}(t^n)\cdot\nabla \textbf{u}(t^n),\textbf{u}(t^n))- \frac{r_h^n}{r_h^{n+1}\sqrt{E_{1,h}^n}}(\textbf{u}_h^n\cdot\nabla \textbf{u}_h^n,\textbf{u}_h^{n+1})\right.\nonumber\\
			&&\left.+\frac{1}{\sqrt{E_1(t^n)}}(\textbf{u}(t^n)\cdot\nabla \phi(t^n),\mu(t^n))-\frac{r_h^n}{r_h^{n+1}\sqrt{E_{1,h}^n}}(\textbf{u}_h^n\cdot\nabla \phi_h^n,\mu_h^{n+1})\right.\nonumber\\
			&&\left.-\left(\frac{1}{\sqrt{E_1(t^n)}}(\mu(t^n)\cdot\nabla \phi(t^n),\textbf{u}(t^n))-\frac{r_h^n}{r_h^{n+1}\sqrt{E_{1,h}^n}}(\mu_h^n\cdot\nabla \phi_h^n,\textbf{u}_h^{n+1})\right)\right) \nonumber\\
			&=&\Delta t(\|\Theta_{\mu h}^{n+1}\|_0^2+\|\Theta_{u h}^{n+1}\|_0^2)
			+\sum_{i=1}^{9}G_i.
		\end{eqnarray}
We next estimate each term on the right-hand side of \eqref{w} separately.
% Let's estimate each term on the right side of equation \eqref{w} one by one.
\begin{eqnarray*}
		|G_1|&=&|2\Delta t((R_u,A_h^{-1}\Phi_{u h}^{n+1})+(R_\phi,\triangle_h^{-1}\Phi_{\mu h}^{n+1})-(\Theta_{\phi h}^{n+1},d_t\Phi_{\phi h}^{n+1})+e_{rh}^{n+1}(R_r-T_1))|\\
        &\leq&\frac{\Delta t}{6}(\|\Phi_{uh}^{n+1}\|_0^2+\|\Phi_{\mu h}^{n+1}\|_0^2)+C\Delta t(\|R_u\|_{-2}^2+\|R_\phi\|_{-2}^2)+\Delta t^2\|d_t\Phi_{\phi h}^{n+1}\|_0^2\\
        &&+\|d_t\Theta_{\phi h}^{n+1}\|_0^2+\Delta t(|e_{rh}^{n+1}|^2+\|T_1\|_0^2+|R_r|^2),\\
			|G_2|&=&|2\Delta t(g(\phi(t^{n+1}))-
			g(\phi_h^n),d_t\triangle_h^{-1}\Phi_{\phi h}^{n+1})|\\
			&=&2\Delta t|(g(\phi(t^{n+1}))-g(\phi(t^n))+g(\phi(t^n))-g(\phi_h^n),d_t\triangle_h^{-1}\Phi_{\phi h}^{n+1})|\\
			&\leq&2\Delta t\|g(\phi(t^{n+1}))-g(\phi(t^n))+g(\phi(t^n))-g(\phi_h^n)\|_0\|d_t\Phi_{\phi h}^{n+1}\|_{-2}.
		\end{eqnarray*}
		%Since by  (\ref{eequ}a), Lemma \ref{truncation} and the fact that $ \|\cdot\|_{-2}\leq\|\cdot\|_{-1} $, we have
        By (\ref{eequ}a), Lemma~\ref{truncation}, and the fact that $ \|\cdot\|_{-2}\leq\|\cdot\|_{-1} $, we have
		\begin{eqnarray}\label{dtphi}
			\|d_t\Phi_{\phi h}^{n+1}\|_{-2}
			&\leq&\|\triangle_h e_{\mu h}^{n+1}\|_{-2}+\|d_t\Theta_{\phi h}^{n+1}\|_{-1}+\|R_\phi^n\|_{-1}\nonumber\\
			&&+\left\|\frac{r(t^{n+1})}
			{\sqrt{E_1(t^{n+1})}}\textbf{u}(t^{n+1})\cdot\nabla\phi(t^{n+1})-\frac{r_h^n}
			{\sqrt{E_{1, h}^n}}\textbf{u}_h^n\cdot\nabla\phi_h^n \right\|_{-1}
			\nonumber\\
			&\leq&\|e_{\mu h}^{n+1}\|_0+C\frac{h^2}{\Delta t^{1/2}}\int_{t^n}^{t^{n+1}}\|\phi_t\|_1dt+\|R_\phi^n\|_{-1}\nonumber\\
			&&+C\Delta t(\|\textbf{u}_t\|_{L^\infty(0,T;L^2(\Omega))}\|\phi(t^{n+1})\|_2+\|A\textbf{u}(t^n)\|_0\|\phi_t\|_{L^\infty(0,T;L^2(\Omega))})\nonumber\\
			&&+C(\|\phi(t^n)\|_2+\|A_h\textbf{u}_h^n\|_0+\|\nabla u_h^n\|_0\|\phi_h^n\|_1)(\| \Phi_{u h}^n\|_0+\|\Theta_{u h}^n\|_0\nonumber\\
			&&+\| \Phi_{\phi h}^n\|_0+\|\Theta_{\phi h}^n\|_0+|e_{rh}^n|).
		\end{eqnarray}
		Then, it follows
		\begin{eqnarray*}
			|G_2|
			&\leq&\frac{\Delta t}{6}\|e_{\mu h}^{n+1}\|_0^2+C\Delta t\|R_\phi^n\|_{-1}^2+C \Delta t^3(\|\textbf{u}_t\|_{L^\infty(0,T;L^2(\Omega))}^2\|\phi(t^{n+1})\|_2^2\\
			&&+\|A\textbf{u}(t^n)\|_0^2\|\phi_t\|_{L^\infty(0,T;L^2(\Omega))}^2+\|\phi_t\|_{L^\infty(0,T;L^2(\Omega))}^2)+C\Delta t(\|\phi(t^n)\|_2^2+\|A_h\textbf{u}_h^n\|_0^2\\
			&&+\|\nabla u_h^n\|_0^2\|\phi_h^n\|_1^2)(\| \Phi_{u h}^n\|_0^2+\|\Theta_{u h}^n\|_0^2+\| \Phi_{\phi h}^n\|_0^2+\|\Theta_{\phi h}^n\|_0^2+|e_{rh}^n|^2)\\
   &&+Ch^{4}\int_{t^n}^{t^{n+1}}\|\phi_t\|_1^2dt.
		\end{eqnarray*}
		
		Next, we estimate the term $ G_3 $,
		\begin{eqnarray*}
			|G_3|&=&\left| \Delta te_{rh}^{n+1}\left( \left(\frac{g(\phi(t^n))}{r(t^n)}-\frac{g(\phi_h^n)}{r_h^{n+1}},d_t\phi_h^{n+1}\right) +\left(  \frac{g(\phi(t^n))}{r(t^n)},d_te_{\phi h}^{n+1}\right) \right)\right| \\
			&=&\left| \Delta te_{rh}^{n+1} \left(\frac{g(\phi(t^n))}{r(t^n)}-\frac{g(\phi_h^n)}{r_h^{n+1}},d_t\phi_h^{n+1}\right) \right|+\left| \Delta te_{rh}^{n+1} \left(  \frac{g(\phi(t^n))}{r(t^n)},d_te_{\phi h}^{n+1}\right)\right|\\
			&=&|B_1|+|B_2|.
		\end{eqnarray*}
For $B_1$, we decompose the coefficient difference into two parts
\begin{equation*}\label{eq:decomp}
\frac{g(\phi(t^n))}{r(t^n)} - \frac{g(\phi^n_h)}{r^{n+1}_h}
= \underbrace{\frac{g(\phi(t^n))}{r(t^n)} - \frac{g(\phi(t^n))}{r^{n+1}_h}}_{I_1}
+ \underbrace{\frac{g(\phi(t^n))}{r^{n+1}_h} - \frac{g(\phi^n_h)}{r^{n+1}_h}}_{I_2}.
\end{equation*}
According to Lemma \ref{lem2.1},
\begin{eqnarray*}
\|I_1 \|_0&=& g(\phi(t^n))\left\|\frac{1}{r(t^n)} - \frac{1}{r^{n+1}_h}\right\|_0\\ 
&=& g(\phi(t^n))\left\|\frac{1}{r(t^n)}-\frac{1}{r(t^{n+1})}+\frac{1}{r(t^{n+1})}-\frac{1}{r_h^{n+1}}\right\|_0\\ 
&=&g(\phi(t^n))\left\|\frac{E_1(t^{n+1})-E_1(t^n)}{r(t^{n+1})r(t^n)(r(t^{n+1})+r(t^n))}+\frac{E_{1,h}^{n+1}-E_1(t^{n+1})}{r_h^{n+1}r(t^{n+1})(r_h^{n+1}+r(t^{n+1}))}\right\|_0\nonumber\\
			&\leq&C(\Delta t\|\phi_t\|_{L^\infty(0,T;L^2(\Omega))}+\|e_{\phi h}^{n+1}\|_0),
\end{eqnarray*}
\begin{equation*}
I_2 = \frac{1}{r^{n+1}_h}\big(g(\phi(t^n)) - g(\phi^n_h)\big).
\end{equation*}
There exists $\zeta^n$ between $\phi(t^n)$ and $\phi^n_h$ such that
\begin{equation*}
g(\phi(t^n)) - g(\phi^n_h) = g'(\zeta^n)\big(\phi(t^n) - \phi^n_h\big),
\end{equation*}
%Since $g'(\phi) = \frac{1}{\varepsilon}\phi(\phi^2-1)$ and $|\phi| \le 1$ a.e.\ by the maximum principle, we have $|g'(\zeta^n)| \le C$. Together with $\frac{1}{|r^{n+1}_h|} \le C$, we obtain
% where we have used the fact that $\zeta^n$ lies between $\phi(t_n)$ and $\phi_h^n$, and both $\phi(t_n)$ and $\phi_h^n$ are uniformly bounded in $L^\infty(\Omega)$ by Assumption \ref{aassume}, the Sobolev embedding and \eqref{s4} together with Lemma 
% \ref{lem5.2}. Therefore, $g'(\zeta^n)$ has a uniform upper bound, together with the bound $\frac{1}{|r_h^{n+1}|}\leq C$, we obtain
where we have used the fact that $\zeta^n$ lies between $\phi(t_n)$ and $\phi_h^n$. By Assumption \ref{aassume}, the Sobolev embedding theorem, \eqref{s4}, and Lemma \ref{lem5.2}, both $\phi(t_n)$ and $\phi_h^n$ are uniformly bounded in $L^\infty(\Omega)$. Consequently, $\zeta^n$ is also uniformly bounded in $L^\infty(\Omega)$, implying that $g'(\zeta^n)$ admits a uniform upper bound. Combining this with the estimate $\frac{1}{|r_h^{n+1}|}\leq C$, we obtain

\begin{equation*}
\|I_2\|_0 \le C\,\|\phi(t^n) - \phi^n_h\|_0 = C\,\|e^n_{\phi h}\|_0.
\end{equation*}

By Lemma \ref{final}, it follows
	\begin{eqnarray*}
	|B_1|
	&\leq&C\Delta t|e_{r h}^{n+1}|^2+C\Delta t(\|e_{\phi h}^{n+1}\|_0^2+\|e_{\phi h}^n\|_0^2)\\
	&&+C\Delta t^3\|\phi_t\|_{L^\infty(0,T;L^2(\Omega))}^2.
\end{eqnarray*}	
By \eqref{dtphi}, it holds 
	\begin{eqnarray*}
	|B_2|&\leq&C\Delta t|e_{rh}^{n+1}|\|d_te_{\phi h}^{n+1}\|_{-2}\\
	&\leq&\frac{\Delta t}{6}\|e_{\mu h}^{n+1}\|_0^2+C\Delta t\|R_\phi^n\|_{-1}^2+C\Delta t|e_{rh}^{n+1}|^2+C \Delta t^3(\|\textbf{u}_t\|_{L^\infty(0,T;L^2(\Omega))}^2\|\phi(t^{n+1})\|_2^2\\
	&&+\|A\textbf{u}(t^n)\|_0^2\|\phi_t\|_{L^\infty(0,T;L^2(\Omega))}^2)+C\Delta t(\|\phi(t^n)\|_2^2+\|A_h\textbf{u}_h^n\|_0^2\\
 &&+\|\nabla u_h^n\|_0^2\|\phi_h^n\|_1^2)(\| \Phi_{u h}^n\|_0^2
 +\|\Theta_{u h}^n\|_0^2+\| \Phi_{\phi h}^n\|_0^2+\|\Theta_{\phi h}^n\|_0^2+|e_{rh}^n|^2).
\end{eqnarray*}
Then,
	\begin{eqnarray*}
	|G_3|
	&\leq&\frac{\Delta t}{6}\|e_{\mu h}^{n+1}\|_0^2+C\Delta t\|R_\phi^n\|_{-1}^2+C \Delta t^3(\|\textbf{u}_t\|_{L^\infty(0,T;L^2(\Omega))}^2\|\phi(t^{n+1})\|_2^2\\
	&&+\|A\textbf{u}(t^n)\|_0^2\|\phi_t\|_{L^\infty(0,T;L^2(\Omega))}^2+\|\phi_t\|_{L^\infty(0,T;L^2(\Omega))}^2)\\
 &&+C\Delta t(\|\phi(t^n)\|_2^2+\|A_h\textbf{u}_h^n\|_0^2+\|\nabla u_h^n\|_0^2\|\phi_h^n\|_1^2)(\| e_{u h}^n\|_0^2+\| e_{\phi h}^n\|_0^2 \\
 &&+\| e_{\phi h}^{n+1}\|_0^2+|e_{rh}^n|^2+|e_{rh}^{n+1}|^2).
\end{eqnarray*}
	
%With \eqref{s1}-\eqref{s4}, the right-hand remaining terms of \eqref{w} are estimated as follows	
Using \eqref{s1}--\eqref{s4}, the remaining terms on the right-hand side of \eqref{w} can be estimated as follows:
\begin{eqnarray*}
|G_4|&=&			
\left| 2\Delta t\left( \frac{r(t^{n+1})}
{\sqrt{E_1(t^{n+1})}}\mu(t^{n+1})\nabla\phi(t^{n+1})-\frac{r_h^n}
{\sqrt{E_{1, h}^n}}\mu_h^n\nabla\phi_h^n,A_h^{-1}\Phi_{u h}^{n+1}\right)\right| \\
&=&2\Delta t \Bigg{|} (\mu(t^{n+1})\cdot\nabla\phi(t^{n+1}),A_h^{-1}\Phi_{u h}^{n+1})-(\mu(t^n)\cdot\nabla\phi(t^n),A_h^{-1}\Phi_{u h}^{n+1})\\
&&\left.+(\mu(t^n)\cdot\nabla\phi(t^n),A_h^{-1}\Phi_{u h}^{n+1})
-(\mu_h^n\cdot\nabla\phi_h^n,A_h^{-1}\Phi_{u h}^{n+1})
+\frac{e_{rh}^n}{\sqrt{E_1(t^n)}}(
\mu_h^n\cdot\nabla\phi_h^n,A_h^{-1}\Phi_{u h}^{n+1})\right.\\
&&\left.
+r_h^n\left( \frac{1}{\sqrt{E_1(t^n)}}-\frac{1}{\sqrt{E_{1,h}^n}}\right) (
\mu_h^n\cdot\nabla\phi_h^n,A_h^{-1}\Phi_{u h}^{n+1})\right|.
\end{eqnarray*}
%Thanks to \eqref{s1}, \eqref{s3} and the Young inequality, we have
By \eqref{s1}, \eqref{s3}, and Young's inequality, we obtain
\begin{eqnarray*}
&&2\Delta t|(\mu(t^{n+1})\cdot\nabla\phi(t^{n+1}),A_h^{-1}\Phi_{u h}^{n+1})-(\mu(t^n)\cdot\nabla\phi(t^n),A_h^{-1}\Phi_{u h}^{n+1})| \\
&=&2\Delta t|((\mu(t^{n+1})-\mu(t^n))\cdot\nabla\phi(t^{n+1})+\mu(t^n)\cdot\nabla(\phi(t^{n+1})-\phi(t^n)),A_h^{-1}\Phi_{u h}^{n+1})|\\
&\leq&\Delta t\|\Phi_{u h}^{n+1}\|_{-1}^2+\Delta t^3(\|\mu_t\|_{L^\infty(0,T;H^1(\Omega))}^2+\|\phi_t\|_{L^\infty(0,T;H^1(\Omega))}^2)(\|\phi(t^{n+1})\|_1^2+\|\mu(t^n)\|_1^2).
\end{eqnarray*}
\begin{eqnarray*}
&&2\Delta t|(\mu(t^n)\cdot\nabla\phi(t^n),A_h^{-1}\Phi_{u h}^{n+1})
-(\mu_h^n\cdot\nabla\phi_h^n,A_h^{-1}\Phi_{u h}^{n+1})| \\
&=&2\Delta t|(e_{\mu h}^n\cdot\nabla\phi(t^n),A_h^{-1}\Phi_{u h}^{n+1})
+(\mu_h^n\cdot\nabla e_{\phi h}^n,A_h^{-1}\Phi_{u h}^{n+1})|
\\
&\leq&\frac{\Delta t}{6}(\|\Phi_{u h}^{n+1}\|_0^2+\| e_{\mu h}^n\|_0^2)
+C\Delta t(\|\phi(t^n)\|_2^2\|\Phi_{uh}^{n+1}\|_{-1}^2+\|\mu_h^n\|_1^2\|\Phi_{\phi h}^n\|_0^2)\\
&&+C\Delta t\|\mu_h^n\|_1^2\| \Theta_{\phi h}^n\|_0^2.
\end{eqnarray*}		
Using Lemma \ref{lem2.1}, it follows		
\begin{eqnarray*}
&&2\Delta t\left| \frac{e_{rh}^n}{\sqrt{E_1(t^n)}}(
\mu_h^n\cdot\nabla\phi_h^n,A_h^{-1}\Phi_{u h}^{n+1})+ r_h^n\left( \frac{1}{\sqrt{E_1(t^n)}}-\frac{1}{\sqrt{E_{1,h}^n}}\right) (
\mu_h^n\cdot\nabla\phi_h^n,\tilde{\Phi}_{u h}^{n+1})\right|\\
&\leq&\frac{\Delta t}{6}\|\Phi_{uh}^{n+1}\|_0^2+C\Delta t\|\mu_h^n\|_0^2\|\nabla\phi_h^n\|_0^2(|e_{rh}^n|^2+\|e_{\phi h}^n\|_0^2),
\end{eqnarray*}		
Then,
\begin{eqnarray*}
|G_5|&=&			
\left| 2\Delta t\left(\frac{r(t^{n+1})}{\sqrt{E_1(t^{n+1})}}\textbf{u}(t^{n+1})\cdot\nabla \textbf{u}(t^{n+1})-\frac{r_h^n}
			{\sqrt{E_{1, h}^n}}\textbf{u}_h^n\cdot\nabla \textbf{u}_h^n,A_h^{-1}\Phi_{u h}^{n+1}\right)\right| \\
&\leq&\Delta t^3\|\textbf{u}_t\|_{L^\infty(0,T;H^1(\Omega))}^2(\|\nabla\textbf{u}(t^{n+1})\|_0^2+\|\nabla\textbf{u}(t^n)\|_0^2)\\
&&+\frac{\Delta t}{6}(\|\Phi_{u h}^{n+1}\|_0^2+\| e_{u h}^n\|_0^2)
+C\Delta t(\|A\textbf{u}(t^n)\|_0^2+\|A_h\textbf{u}_h^n\|_0^2)\|\Phi_{u h}^{n+1}\|_{-1}^2\\
&&+C\Delta t\|\textbf{u}_h^n\|_0^2\|\nabla\textbf{u}_h^n\|_0^2(|e_{rh}^n|^2+\|e_{\phi h}^n\|_0^2),
\end{eqnarray*}	
 $G_6$  is estimated as follows
		\begin{eqnarray*}
			|G_6|&=&\left| 2\Delta t\left( \frac{r(t^{n+1})}
			{\sqrt{E_1(t^{n+1})}}\textbf{u}(t^{n+1})\cdot\nabla\phi(t^{n+1})-\frac{r_h^n}
			{\sqrt{E_{1, h}^n}}\textbf{u}_h^n\cdot\nabla\phi_h^n,\triangle^{-1}_h\Phi_{\mu h}^{n+1}\right) \right| \\
			&=&2\Delta t \Bigg{|}  (\textbf{u}(t^{n+1})\cdot\nabla\phi(t^{n+1}),\triangle^{-1}_h\Phi_{\mu h}^{n+1})-(\textbf{u}(t^n)\cdot\nabla\phi(t^n),\triangle^{-1}_h\Phi_{\mu h}^{n+1})\\
		&&\left.+(\textbf{u}(t^n)\cdot\nabla\phi(t^n),\triangle^{-1}_h\Phi_{\mu h}^{n+1})
		-(\textbf{u}_h^n\cdot\nabla\phi_h^n,\triangle^{-1}_h\Phi_{\mu h}^{n+1})\right.\\
		&&\left.+\frac{e_{rh}^n}{\sqrt{E_1(t^n)}}(
		\textbf{u}_h^n\cdot\nabla\phi_h^n,\triangle^{-1}_h\Phi_{\mu h}^{n+1})\right.\\
		&&\left.+r_h^n\left( \frac{1}{\sqrt{E_1(t^n)}}-\frac{1}{\sqrt{E_{1,h}^n}}\right) (
		\textbf{u}_h^n\cdot\nabla\phi_h^n,\triangle^{-1}_h\Phi_{\mu h}^{n+1})\right|.
		\end{eqnarray*}
        
		By \eqref{s1} and \eqref{s3}, we can derive
		\begin{eqnarray*}
		&&2\Delta t|(\textbf{u}(t^{n+1})\cdot\nabla\phi(t^{n+1}),\triangle^{-1}_h\Phi_{\mu h}^{n+1})-(\textbf{u}(t^n)\cdot\nabla\phi(t^n),\triangle^{-1}_h\Phi_{\mu h}^{n+1})| \\
		&=&2\Delta t|((\textbf{u}(t^{n+1})-\textbf{u}(t^n))\cdot\nabla\phi(t^{n+1})+\textbf{u}(t^n)\cdot\nabla(\phi(t^{n+1})-\phi(t^n)),\triangle^{-1}_h\Phi_{\mu h}^{n+1})|\\
		&\leq&\frac{\Delta t}{5}\|
		\Phi_{\mu h}^{n+1}\|_0^2+C\Delta t^3(\|\textbf{u}_t\|_{L^\infty(0,T;L^2(\Omega))}^2\|\phi(t^{n+1})\|_1^2
		+\|\nabla\textbf{u}(t^n)\|_0^2\|\phi_t\|_{L^\infty(0,T;L^2(\Omega))}^2),
  \end{eqnarray*}
		\begin{eqnarray*}
		&&2\Delta t|(\textbf{u}(t^n)\cdot\nabla\phi(t^n),\triangle^{-1}_h\Phi_{\mu h}^{n+1})
		-(\textbf{u}_h^n\cdot\nabla\phi_h^n,\triangle^{-1}_h\Phi_{\mu h}^{n+1})| \\
		&=&2\Delta t|(e_{uh}^n\cdot\nabla\phi(t^n)
		+\textbf{u}_h^n\cdot\nabla e_{\phi h}^n,\triangle^{-1}_h\Phi_{\mu h}^{n+1})| \\
		&=&2\Delta t|((\Theta_{uh}^n+\Phi_{uh}^n)\cdot\nabla\phi(t^n)
		+\textbf{u}_h^n\cdot\nabla e_{\phi h}^n,\triangle^{-1}_h\Phi_{\mu h}^{n+1})| \\
		&\leq&\frac{\Delta t}{6}\|
		\Phi_{\mu h}^{n+1}\|_0^2+C\Delta t(\|\phi(t^n)\|_1^2(\|\Theta_{uh}^n\|_0^2+\|\Phi_{uh}^n\|_{-1}^2)+\|\nabla\textbf{u}_h^n\|_0^2\|e_{\phi h}^n\|_0^2),
  \end{eqnarray*}
		\begin{eqnarray*}
		&&2\Delta t\left| \frac{e_{rh}^n}{\sqrt{E_1(t^n)}}(
		\textbf{u}_h^n\cdot\nabla\phi_h^n,\triangle^{-1}_h\Phi_{\mu h}^{n+1})+ r_h^n\left( \frac{1}{\sqrt{E_1(t^n)}}-\frac{1}{\sqrt{E_{1,h}^n}}\right) (
		\textbf{u}_h^n\cdot\nabla\phi_h^n,\triangle^{-1}_h\Phi_{\mu h}^{n+1})\right|\\
		&\leq&\frac{\Delta t}{6}\|
		\Phi_{\mu h}^{n+1}\|_0^2+C\Delta t\|\textbf{u}_h^n\|_0^2\|\nabla\phi_h^n\|_0^2(|e_{rh}^n|^2+\|e_{\phi h}^n\|_0^2).
  \end{eqnarray*}
%The last three terms on the right-hand side of \eqref{w} can be estimated as
The last three terms on the right-hand side of \eqref{w} can be estimated as follows:
\begin{eqnarray*}
			|G_7|&=&\left|\Delta te_{rh}^{n+1}\left(  \frac{1}{\sqrt{E_1(t^n)}}(\textbf{u}(t^n)\cdot\nabla \phi(t^n),\mu(t^n))-\frac{r_h^n}{r_h^{n+1}\sqrt{E_{1,h}^n}}(\textbf{u}_h^n\cdot\nabla \phi_h^n,\mu_h^{n+1})\right) \right| \\
            &=&\frac{\Delta t|e_{rh}^{n+1}|}{\sqrt{E_1(t^n)}}|(e_{uh}^n\cdot\nabla \phi(t^n),\mu(t^n))+(\textbf{u}_h^n\cdot\nabla e_{\phi h}^n,\mu(t^n))+(\textbf{u}_h^n\cdot\nabla \phi_h^n,\mu(t^n)-\mu(t^{n+1}))\\
			&&\ \ \ \ \ \ \ \ \ \ \ \ \ \ \
			+(\textbf{u}_h^n\cdot\nabla \phi_h^n,e_{\mu h}^{n+1})|\\
			&&+\frac{\Delta t|e_{r h}^{n+1}|}{r_h^{n+1}}\left(\frac{r_h^{n+1}}{\sqrt{E_1(t^n)}}-\frac{r_h^n}{\sqrt{E_{1,h}^n}} \right)(\textbf{u}_h^n\cdot\nabla \phi_h^n,\mu_h^{n+1})\\
			&\leq&\frac{\Delta t}{6}\|e_{\mu h}^{n+1}\|_0^2+\frac{\Delta t}{4}\|e_{u h}^{n}\|_0^2+C\Delta t^3(\|\mu_t\|_{L^\infty(0,T;H^1(\Omega))}^2+|r_t|^2)\|\nabla\textbf{u}_h^n\|_0^2\|\nabla\phi_h^n\|_0^2\\
			&&+C\Delta t(|e_{rh}^{n+1}|^2+|e_{rh}^n|^2)(\|\nabla\textbf{u}_h^n\|_0^2\|\triangle_h\phi_h^n\|_0^2+\|\mu_h^{n+1}\|_1^2+1)\\
			&&
			+C\Delta t(\|\phi(t^n)\|_2^2\|\mu(t^n)\|_1^2+\|\nabla \textbf{u}_h^n\|_0^2\|\mu(t^n)\|_2^2+\|\nabla\textbf{u}_h^n\|_0^2\|\triangle_h\phi_h^n\|_0^2)\|e_{\phi h}^n\|_0^2,
	\end{eqnarray*}
    \begin{eqnarray*}
|G_8| &=& \left| \Delta t e_{rh}^{n+1} \left( \frac{1}{\sqrt{E_1(t^n)}} (\textbf{u}(t^n) \cdot \nabla \textbf{u}(t^n),\textbf{u}(t^n)) - \frac{r_h^n}{r_h^{n+1}\sqrt{E_{1,h}^n}} (\textbf{u}_h^n \cdot \nabla \textbf{u}_h^n, \textbf{u}_h^{n+1}) \right) \right| \\
&=& \Delta t |e_{rh}^{n+1}| \left| \frac{1}{\sqrt{E_1(t^n)}} (\textbf{u}(t^n) \cdot \nabla \textbf{u}(t^n), \textbf{u}(t^n)) - \frac{r_h^n}{r_h^{n+1}\sqrt{E_{1,h}^n}} (\textbf{u}_h^n \cdot \nabla \textbf{u}_h^n, \textbf{u}_h^{n+1}) \right| \\
&=& \frac{\Delta t |e_{rh}^{n+1}|}{\sqrt{E_{1,h}(t^n)}} \Big| (e_{uh}^n \cdot \nabla \textbf{u}(t^n), \textbf{u}(t^n)) + (\textbf{u}_h^n \cdot \nabla e_{uh}^n, \textbf{u}(t^n)) + (\textbf{u}_h^n \cdot \nabla \textbf{u}_h^n, \textbf{u}(t^n) - \textbf{u}(t^{n+1})) \\
&&\qquad\qquad\qquad + (\textbf{u}_h^n \cdot \nabla \textbf{u}_h^n, e_{uh}^{n+1}) \Big| \\
&& + \frac{\Delta t |e_{rh}^{n+1}|}{r_h^{n+1}} \left( \frac{r_h^{n+1}}{\sqrt{E_1(t^n)}} - \frac{r_h^n}{\sqrt{E_{1,h}^n}} \right) (\textbf{u}_h^n \cdot \nabla \textbf{u}_h^n, \textbf{u}_h^{n+1}) \\
&\leq& \frac{\Delta t}{4}\|e_{u h}^{n}\|_0^2+\frac{\Delta t}{4}\|e_{u h}^{n+1}\|_0^2+C\Delta t^3 \left( \|\textbf{u}_t\|_{L^\infty(0,T;H^1(\Omega))}^2 + |r_t|^2 \|\textbf{u}_h^n\|_0^2 \|\nabla \textbf{u}_h^n\|_0^2 \right)\\
&&+ C\Delta t \Big( \|A\textbf{u}(t^n)\|_0^2 ( \|\nabla\textbf{u}(t^n)\|_0^2 + \|\nabla \textbf{u}_h^n\|_0^2)  + \|\textbf{u}_h^n\|_0^2 \|\nabla\textbf{u}_h^n\|_0^2 + \|\textbf{u}_h^n\|_1^2 \|A_h \textbf{u}_h^n\|_0^2  \Big) \left( |e_{rh}^{n+1}|^2 + |e_{rh}^n|^2 \right. \\
&& \left. + \|\Theta_{\phi h}^n\|_0^2+ \|\Phi_{\phi h}^n\|_0^2   \right),
\end{eqnarray*}
\begin{eqnarray*}
|G_9| &= &\left| \Delta t e_{rh}^{n+1} \left( \frac{1}{\sqrt{E_1(t^n)}} (\mu(t^n) \cdot \nabla \phi(t^n), \textbf{u}(t^n)) - \frac{r_h^n}{r_h^{n+1}\sqrt{E_{1,h}^n}} (\mu_h^n \cdot \nabla \phi_h^n,  \textbf{u}_h^{n+1}) \right) \right| \\
&\leq& C\Delta t^3 \left( \|\textbf{u}_t\|_{L^\infty(0,T;H^1(\Omega))}^2 + |r_t|^2 \|\mu_h^n\|_0^2 \|\nabla \phi_h^n\|_0^2 \right) + C\Delta t \Big( \|A\textbf{u}(t^n)\|_0^2 (\|\nabla \phi(t^n)\|_0^2 \\
&& + \|\mu_h^n\|_1^2) + \|\mu_h^n\|_1^2 \|\nabla \phi_h^n\|_0^2 + \|\mu_h^n\|_1^2 \|\phi_h^n\|_2^2 + \|\nabla u_h^{n+1}\|_0^2 \Big) \left( |e_{rh}^{n+1}|^2 + |e_{rh}^n|^2 \right. \\
&& \left. + \|\theta_{\phi h}^n\|_0^2 + \|\Phi_{\phi h}^n\|_0^2 \right) + \frac{\Delta t}{4}\|e_{uh}^{n+1}\|_0^2.
\end{eqnarray*}

		%Substituting the above inequalities into \eqref{w}, summing from $ n = 0 $ to $ m $, using \eqref{orth2}, \eqref{orth4}, Lemmas \ref{truncation}, \ref{lem5.1}-\ref{lem5.3} and the triangle inequality, we finish the proof.
        Substituting the above inequalities into \eqref{w}, summing the resulting inequality from $n=0$ to $m$, and applying \eqref{orth2}, \eqref{orth4}, Lemma~\ref{truncation}, Lemmas~\ref{lem5.1}--\ref{lem5.3}, and the triangle inequality, we obtain the desired result.
	\end{proof}

\begin{thm}\label{3.6}
	For all $ m \geq 0 $, under the assumptions of Theorem \ref{thm5.5} and $e_{u h}^0=0$, it holds
	\begin{eqnarray*}
		\|e_{u h}^{m+1}\|_0^2
		+\Delta t\sum_{n=0}^m\|d_te_{u h}^{n+1}\|_{-1}^2
		\leq C(\Delta t^2+h^4).
	\end{eqnarray*}
\end{thm}

\begin{proof}
	Taking $\textbf{v}_h=-2\Delta tA^{-1}_hd_te_{u h}^{n+1}$ in (\ref{eequ}c), it holds
	\begin{eqnarray}\label{eo1}
		&&\|e_{u h}^{n+1}\|_0^2-\|e_u^n\|_0^2
		+\|e_{u h}^{n+1}-e_u^n\|_0^2
		+2\Delta t\|d_te_{u h}^{n+1}\|_{-1}^2\nonumber\\
		&=&2\Delta t\left( \frac{r(t^{n+1})}{\sqrt{E_1(t^{n+1})}}\textbf{u}(t^{n+1})\cdot\nabla \textbf{u}(t^{n+1})-\frac{r_h^n}{\sqrt{E_{1,h}^n}}\textbf{u}_h^n\cdot\nabla \textbf{u}_h^n,A_h^{-1} d_t e_{uh}^{n+1}\right) \nonumber\\
			&&-2\Delta t\left( \frac{r(t^{n+1})}{\sqrt{E_1(t^{n+1})}}\mu(t^{n+1})\nabla\phi(t^{n+1})
			-\frac{r_h^n}{\sqrt{E_{1,h}^n}}\mu_h^n\nabla\phi_h^n,A_h^{-1} d_t e_{uh}^{n+1}\right)\nonumber\\
            &&-2\Delta t(R_u, A_h^{-1} d_t e_{uh}^{n+1})\nonumber\\
		&=&P_1+P_2+P_3.
	\end{eqnarray}
	We estimate the terms on the right-hand side of \eqref{eo1} as follows:
	%We can treat the right-hand terms of \eqref{eo1} as follows
	\begin{eqnarray*}
		|P_1| &=& \left| 2\Delta t \left( \frac{r(t^{n+1})}{\sqrt{E_1(t^{n+1})}} \textbf{u}(t^{n+1}) \cdot \nabla \textbf{u}(t^{n+1}) - \frac{r_h^n}{\sqrt{E_{1,h}^n}} \textbf{u}_h^n \cdot \nabla \textbf{u}_h^n, A_h^{-1} d_t e_{uh}^{n+1} \right) \right| \\
&=& 2\Delta t \left| \left( (\textbf{u}(t^{n+1}) - \textbf{u}(t^n)) \cdot \nabla \textbf{u}(t^{n+1}) + \textbf{u}(t^n) \cdot \nabla(\textbf{u}(t^{n+1}) - \textbf{u}(t^n)), A_h^{-1} d_t e_{uh}^{n+1} \right) \right. \\
&& + (e_{uh}^n \cdot \nabla \textbf{u}(t^n), A_h^{-1} d_t e_{uh}^{n+1}) + (\textbf{u}_h^n \cdot \nabla e_{uh}^n, A_h^{-1} d_t e_{uh}^{n+1}) \\
&& + \frac{e_{rh}^n}{\sqrt{E_1(t^n)}} (\textbf{u}_h^n \cdot \nabla \textbf{u}_h^n, A_h^{-1} d_t e_{uh}^{n+1}) \\
&& \left. + r_h^n \left( \frac{1}{\sqrt{E_1(t^n)}} - \frac{1}{\sqrt{E_{1,h}^n}} \right) (\textbf{u}_h^n \cdot \nabla \textbf{u}_h^n, A_h^{-1} d_t e_{uh}^{n+1}) \right|\\
&\leq& \frac{\Delta t}{3} \|d_t e_{uh}^{n+1}\|_{-1}^2 + C\Delta t^3 \|\textbf{u}_t\|_{L^\infty(0,T;H^1(\Omega))}^2 \left( \|\nabla \textbf{u}(t^{n+1})\|_0^2 + \|\nabla \textbf{u}(t^n)\|_0^2 \right) \\
&& + C\Delta t \left( \|A\textbf{u}(t^n)\|_0^2 + \|A_h \textbf{u}_h^n\|_0^2 + \|\nabla \textbf{u}_h^n\|_0^4 \right) \left( |e_{rh}^n|^2 + \|e_{\phi h}^n\|_0^2 + \|e_{uh}^n\|_0^2 \right).\\
		|P_2| &=& \left| 2\Delta t \left( \frac{r(t^{n+1})}{\sqrt{E_1(t^{n+1})}} \mu(t^{n+1}) \nabla\phi(t^{n+1}) - \frac{r_h^n}{\sqrt{E_{1,h}^n}} \mu_h^n \nabla\phi_h^n, A_h^{-1} d_t e_{uh}^{n+1} \right) \right| \\
&=& 2\Delta t \left| \left( (\mu(t^{n+1}) - \mu(t^n)) \cdot \nabla\phi(t^{n+1}) + \mu(t^n) \cdot \nabla(\phi(t^{n+1}) - \phi(t^n)), A_h^{-1} d_t e_{uh}^{n+1} \right) \right. \\
&& + (e_{\mu h}^n \cdot \nabla\phi(t^n), A_h^{-1} d_t e_{uh}^{n+1}) + (\mu_h^n \cdot \nabla e_{\phi h}^n, A_h^{-1} d_t e_{uh}^{n+1}) \\
&& + \frac{e_{rh}^n}{\sqrt{E_1(t^n)}} (\mu_h^n \cdot \nabla\phi_h^n, A_h^{-1} d_t e_{uh}^{n+1}) \\
&& \left. + r_h^n \left( \frac{1}{\sqrt{E_1(t^n)}} - \frac{1}{\sqrt{E_{1,h}^n}} \right) (\mu_h^n \cdot \nabla\phi_h^n, A_h^{-1} d_t e_{uh}^{n+1}) \right|\\
		&\leq& \frac{\Delta t}{3} \|d_t e_{uh}^{n+1}\|_{-1}^2 + C\Delta t^3 \left( \|\mu_t\|_{L^\infty(0,T;H^1(\Omega))}^2 + \|\phi_t\|_{L^\infty(0,T;H^1(\Omega))}^2 \right) \left( \|\phi(t^{n+1})\|_2^2 + \|\mu(t^n)\|_1^2 \right) \\
&& + C\Delta t \left( \|\phi(t^n)\|_2^2 \|e_{\mu h}^n\|_0^2 + (\|\mu_h^n\|_1^2 + \|\mu_h^n\|_1^2 \|\phi_h^n\|_2^2) \left( \|e_{\phi h}^n\|_0^2 + |e_{rh}^n|^2 \right) \right).\\
|P_3| &=& |2\Delta t (R_u, A_h^{-1} d_t e_{uh}^{n+1})| \leq \frac{\Delta t}{3} \|d_t e_{uh}^{n+1}\|_{-1}^2 + C\Delta t \|R_u\|_{-1}^2.
	\end{eqnarray*}
	%Combining above inequalities with \eqref{eo1} and summing from $ n = 0 $ to $ m $, using Lemma \ref{lem1}-\ref{lem3}, \ref{lem5.1}-\ref{lem5.3}  and Theorem \ref{thm5.5}, we completes the proof.
    % Combining the above inequalities with \eqref{eo1} and summing over $n=0,\dots,m$, and using Theorem \ref{thm5.5}, we complete the proof.
    Combining the above inequalities with \eqref{eo1}, summing over $n=0,\dots,m$, and applying Theorem~\ref{thm5.5}, we complete the proof.

\end{proof}

%In the following process of proving the error estimate of pressure, we will encounter a problem term $ \|e_u^{n+1}-e_u^n\|_{-1} $, so we give the following theorem.
In the subsequent analysis of the pressure error estimate, the term $|e_u^{n+1}-e_u^n|_{-1}$ arises. To handle this term, we establish the following lemma.
	\begin{lem}\label{lemp}
		Under \Cref{aassume},
		for all $m\geq0$ and $e_{u h}^0=0$, it holds
		\begin{eqnarray*}
			\sum_{n=0}^m\|e_u^{n+1}-e_u^n\|_{-1}^2\leq
			C\Delta t(\Delta t^2+h^{2}).
		\end{eqnarray*}
	\end{lem}
\begin{proof}
		Taking $\textbf{v}_h=\Delta tA_h^{-1}(e_{uh}^{n+1}-e_{uh}^n)$ in (\ref{eequ}c), we have		
		\begin{eqnarray}\label{eo}
			&&\|e_{uh}^{n+1}-e_{uh}^n\|_{-1}^2-(\triangle_h e_{uh}^{n+1},\Delta tA_h^{-1}(e_{uh}^{n+1}-e_{uh}^n))\nonumber\\
			&=&\Delta t\left( 
			\left( \frac{r(t^{n+1})}{\sqrt{E_1(t^{n+1})}}\mu(t^{n+1})\nabla\phi(t^{n+1})-\frac{r_h^n}
			{\sqrt{E_{1, h}^n}}\mu_h^n\nabla\phi_h^n,A_h^{-1}(e_{uh}^{n+1}-e_{uh}^n)\right.\right) \nonumber\\
			&&-\left.\left( \frac{r(t^{n+1})}{\sqrt{E_1(t^{n+1})}}\textbf{u}(t^{n+1})\cdot\nabla \textbf{u}(t^{n+1})-\frac{r_h^n}
			{\sqrt{E_{1, h}^n}}\textbf{u}_h^n\cdot\nabla \textbf{u}_h^n,A_h^{-1}(e_{uh}^{n+1}-e_{uh}^n)\right) \right) 
			\nonumber\\
			&&+\Delta t(R_u,A_h^{-1}(e_{uh}^{n+1}-e_{uh}^n)).
		\end{eqnarray}
		
		For the second term on the left side of \eqref{eo}, it follows
		\begin{eqnarray*}
			&&-(\triangle_h e_{uh}^{n+1},\Delta tA_h^{-1}(e_{uh}^{n+1}-e_{uh}^n))\\
            &=&(e_{uh}^{n+1},-\Delta t\triangle_h A_h^{-1}(e_{uh}^{n+1}-e_{uh}^n))=\Delta t(e_{uh}^{n+1}, e_{uh}^{n+1}-e_{uh}^n)\\
			&=&\frac{\Delta t}{2}(\|e_{uh}^{n+1}\|_0^2-\|e_{uh}^n\|_0^2+\|e_{uh}^{n+1}-e_{uh}^n\|_0^2).
		\end{eqnarray*}
		Then, \eqref{eo} yields
		\begin{eqnarray}\label{eoo}
			&&\|e_{uh}^{n+1}-e_{uh}^n\|_{-1}^2+\frac{\Delta t}{2}(\|e_{uh}^{n+1}\|_0^2-\|e_{uh}^n\|_0^2+\|e_{uh}^{n+1}-e_{uh}^n\|_0^2)\nonumber\\
			&=&\Delta t(R_u,A_h^{-1}(e_{uh}^{n+1}-e_{uh}^n))\nonumber\\
			&&+\Delta t\left( 
			\left( \frac{r(t^{n+1})}{\sqrt{E_1(t^{n+1})}}\mu(t^{n+1})\nabla\phi(t^{n+1})-\frac{r_h^n}
			{\sqrt{E_{1, h}^n}}\mu_h^n\nabla\phi_h^n,A_h^{-1}(e_{uh}^{n+1}-e_{uh}^n)\right.\right) \nonumber\\
			&&-\left.\left( \frac{r(t^{n+1})}{\sqrt{E_1(t^{n+1})}}\textbf{u}(t^{n+1})\cdot\nabla \textbf{u}(t^{n+1})-\frac{r_h^n}
			{\sqrt{E_{1, h}^n}}\textbf{u}_h^n\cdot\nabla \textbf{u}_h^n,A_h^{-1}(e_{uh}^{n+1}-e_{uh}^n)\right) \right) .
		\end{eqnarray}
		For terms on the right hand side of \eqref{eoo}, we have
%		The right hand side terms of \eqref{eoo} can be treated as follows
		\begin{eqnarray*}
			|\Delta t(R_u,A_h^{-1}(e_{uh}^{n+1}-e_{uh}^n))|&\leq&\frac{1}{6}\|e_{uh}^{n+1}-e_{uh}^n\|_{-1}^2+C\Delta t^2\|R_u^n\|_{-1}^2,
		\end{eqnarray*}
		\begin{eqnarray*}
			&&\left| \Delta t\left(\frac{r(t^{n+1})}{\sqrt{E_1(t^{n+1})}}\mu(t^{n+1})\nabla\phi(t^{n+1})-\frac{r_h^n}
			{\sqrt{E_{1, h}^n}}\mu_h^n\nabla\phi_h^n,A_h^{-1}(e_{uh}^{n+1}-e_{uh}^n)\right)\right| \\
            &=&\Delta t \left| (\mu(t^{n+1})\cdot\nabla\phi(t^{n+1}),A_h^{-1}(e_{uh}^{n+1}-e_{uh}^n))-(\mu(t^n)\cdot\nabla\phi(t^n),A_h^{-1}(e_{uh}^{n+1}-e_{uh}^n))\right.\\
			&&\left.+(e_{\mu h}^n\cdot\nabla\phi(t^n),A_h^{-1}(e_{uh}^{n+1}-e_{uh}^n))
			+(\mu_h^n\cdot\nabla e_{\phi h}^n,A_h^{-1}(e_{uh}^{n+1}-e_{uh}^n))\right.\\
			&&+\frac{e_{rh}^n}{\sqrt{E_1(t^n)}}(
			\mu_h^n\cdot\nabla\phi_h^n,A_h^{-1}(e_{uh}^{n+1}-e_{uh}^n))\\
			&&\left.
			+r_h^n\left( \frac{1}{\sqrt{E_1(t^n)}}-\frac{1}{\sqrt{E_{1,h}^n}}\right) (
			\mu_h^n\cdot\nabla\phi_h^n,A_h^{-1}(e_{uh}^{n+1}-e_{uh}^n))\right|\\
            &\leq&\frac{1}{6}\|e_{uh}^{n+1}-e_{uh}^n\|_{-1}^2+C\Delta t^4(\|\mu_t\|_{L^\infty(0,T;H^1(\Omega))}^2\|\phi(t^{n+1})\|_1^2+\|\mu(t^n)\|_1^2\|\phi_t\|_{L^\infty(0,T;H^1(\Omega))}^2)\\
            &&+C\Delta t^2(\|\phi(t^n)\|_1^2\|e_{\mu h}^n\|_1^2+\|\mu_h^n\|_1^2\|\phi_h^n\|_1^2(|e_{rh}^n|^2+\|e_{\phi h}^n\|_0^2)). 
   \end{eqnarray*}
		Similarly,
		\begin{eqnarray*}
			&&	\left| \Delta t\left( \frac{r(t^{n+1})}{\sqrt{E_1(t^{n+1})}}\textbf{u}(t^{n+1})\cdot\nabla \textbf{u}(t^{n+1})-\frac{r_h^n}
			{\sqrt{E_{1, h}^n}}\textbf{u}_h^n\cdot\nabla \textbf{u}_h^n,A_h^{-1}(e_{uh}^{n+1}-e_{uh}^n)\right) \right| \\
			&=&\Delta t\Big{|}(\textbf{u}(t^{n+1})\cdot\nabla\textbf{u}(t^{n+1}),A_h^{-1}(e_{uh}^{n+1}-e_{uh}^n))-(\textbf{u}(t^n)\cdot\nabla\textbf{u}(t^n),A_h^{-1}(e_{uh}^{n+1}-e_{uh}^n))\\
			&&+(e_{uh}^n\cdot\nabla\textbf{u}(t^n),A_h^{-1}(e_{uh}^{n+1}-e_{uh}^n))
			+(\textbf{u}_h^n\cdot\nabla e_{uh}^n,A_h^{-1}(e_{uh}^{n+1}-e_{uh}^n))\\
			&&+\frac{e_{rh}^n}{\sqrt{E_1(t^n)}}(
			\textbf{u}_h^n\cdot\nabla\textbf{u}_h^n,A_h^{-1}(e_{uh}^{n+1}-e_{uh}^n))\\
			&&+r_h^n\left( \frac{1}{\sqrt{E_1(t^n)}}-\frac{1}{\sqrt{E_{1,h}^n}}\right) (
			\textbf{u}_h^n\cdot\nabla\textbf{u}_h^n,A_h^{-1}(e_{uh}^{n+1}-e_{uh}^n))\Big{|}\\
			&\leq&\frac{1}{6}\|e_{uh}^{n+1}-e_{uh}^n\|_{-1}^2+C\Delta t^4\|\textbf{u}_t\|_{L^\infty(0,T;H^1(\Omega))}^2(\|\nabla\textbf{u}(t^{n+1})\|_0^2+\|\nabla\textbf{u}(t^n)\|_0^2)\\
  &&+C\Delta t^2(\|\nabla\textbf{u}(t^n)\|_0^2+\|\nabla\textbf{u}_h^n\|_0^2)\| e_{uh}^n\|_1^2+\| \textbf{u}_h^n\|_0^2\|\nabla\textbf{u}_h^n\|_0^2(|e_{rh}^n|^2+\|e_{\phi h}^n\|_0^2).
\end{eqnarray*}  
		% Combining the above inequalities with  \eqref{eoo} and summing from $ n = 0 $ to $ m $, using Lemma \ref{truncation}, Lemma \ref{thm4.1}-Lemma \ref{lemu}, we derive the desired results.
        Combining the above inequalities with \eqref{eoo}, summing the resulting inequality from $n=0$ to $m$, and applying Lemma~\ref{truncation} and Lemmas~\ref{thm4.1}--\ref{lemu}, we obtain the desired result.
	\end{proof}

\begin{thm} 
	For all $ m \geq 0 $, it holds
	\begin{eqnarray*}
		\|\phi(t^n)-\phi_h^n\|_0^2+\|\textbf{u}(t^n)-\textbf{u}_h^n\|_0^2
		+|r(t^n)-r_h^n|^2+\Delta t\sum_{n=0}^m(\|\mu(t^n)-\mu_h^n\|_0^2)
		\leq C(\Delta t^2+h^4),\\
		\Delta t\sum_{n=0}^m\|p_h^n-p(t^n)\|_0^2
		\leq C(\Delta t^2+h^2).
	\end{eqnarray*}
\end{thm}
\begin{proof}
	%According to the Theorem \ref{thm5.5}, \ref{3.6} and the triangle inequality, we have
    By Theorem~\ref{thm5.5}, Theorem~\ref{3.6}, and the triangle inequality, we obtain
	\begin{eqnarray*}
		\|\phi(t^n)-\phi_h^n\|_0^2+\|\textbf{u}(t^n)-\textbf{u}_h^n\|_0^2
		+|r(t^n)-r_h^n|^2+\Delta t\sum_{n=0}^m(\|\mu(t^n)-\mu_h^n\|_0^2)
		\leq C(\Delta t^2+h^4).
	\end{eqnarray*}
	By \eqref{eequ}c and the inf-sup condition \eqref{inf-sup}, we obtain
	%From (\ref{eequ}c) with the help of \eqref{inf-sup}, we have
	\begin{eqnarray*}
		\beta\|e_{ph}^{n+1}\|_0&\leq&\frac{(e_{ph}^{n+1},\nabla\cdot \textbf{v}_h)}{\|\nabla \textbf{v}_h\|_0}\\
		&=&\left[ (d_te_{uh}^{n+1},\textbf{v}_h) -(R_u,\textbf{v}_h)
			+(\nabla e_{uh}^{n+1},\nabla \textbf{v}_h)\right.\\
			&&+\left( \frac{r(t^{n+1})}{\sqrt{E_1(t^{n+1})}}\textbf{u}(t^{n+1})\cdot\nabla \textbf{u}(t^{n+1})-\frac{r_h^n}{\sqrt{E_{1,h}^n}}\textbf{u}_h^n\cdot\nabla \textbf{u}_h^n,\textbf{v}_h\right) \\
			&&-\left.
			\left( \frac{r(t^{n+1})}{\sqrt{E_1(t^{n+1})}}\mu(t^{n+1})\nabla\phi(t^{n+1})
			-\frac{r_h^n}{\sqrt{E_{1,h}^n}}\mu_h^n\nabla\phi_h^n,\textbf{v}_h\right) \right]/\|\nabla \textbf{v}_h\|_0\\
			&\leq&\frac{1}{\Delta t}\|e_{uh}^{n+1}-e_{uh}^n\|_{-1}+\|\nabla e_{uh}^{n+1}\|_0+\|R_u\|_{-1}+C\Delta t(\|\mu_t\|_{L^\infty(0,T;H^1(\Omega))}^2\\
&&+\|\phi_t\|_{L^\infty(0,T;H^1(\Omega))}^2+\|\textbf{u}_t\|_{L^\infty(0,T;L^2(\Omega))}^2)(\|\nabla\phi(t^{n+1})\|_0^2+\|\mu(t^n)\|_1^2+\|A\textbf{u}(t^{n+1})\|_0^2)\\
    &&+C\| e_{\mu h}^n\|_1^2
	+C\Delta t(\|\|\mu_h^n\|_1^2\|\nabla\phi_h^n\|_0^2+\|\textbf{u}_h^n\|_0^2\|\nabla\textbf{u}_h^n\|_0^2+\|\phi(t^n)\|_2^2+\|\nabla\textbf{u}(t^n)\|_0^2)(\|e_{\phi h}^n\|_1^2\\
    &&+\|e_{uh}^{n+1}\|_0^2+|e_{rh}^n|^2).
	\end{eqnarray*}
  Multiplying both sides of the inequality by $\Delta t$, summing over $n=0,\dots,m$, and applying Lemma~\ref{lemp}, Lemma~\ref{lem5.1}, and Theorems~\ref{thm5.5}--\ref{3.6}, we complete the proof.
\end{proof}

\section{Numerical experiment}\label{secNum}
	% In this section, we give several examples to test and verify the reliability and efficiency of the proposed numerical scheme for the following CHNS system
In this section, we present several numerical experiments to demonstrate the reliability and efficiency of the proposed numerical scheme for the following CHNS system:
\begin{equation}
\left\{\begin{aligned}
	\phi_t+ \boldsymbol{u}\cdot\nabla\phi -M\Delta \mu& =  0, \quad \text{in}\ \Omega \times [0, T],\\
	\mu+\lambda\Delta\phi-F^{'}(\phi)&=0, \quad \text{in}\ \Omega \times [0, T],\\
	\boldsymbol{u}_t-\nu\Delta \boldsymbol{u} + (\boldsymbol{u}\cdot\nabla)\boldsymbol{u} + \nabla p - \mu \nabla \phi& = \boldsymbol{g}, \quad \text{in}\ \Omega \times [0, T],\\
	\nabla \cdot \boldsymbol{u} &= 0, \quad\, \text{in}\ \Omega \times [0, T],\\
	\nabla \phi\cdot\mathbf{n}=\nabla \mu\cdot\mathbf{n}=0,\mathbf{u} & = \mathbf{0},\quad\, \text{on}\  \partial\Omega\times(0,T],
\end{aligned}
\right.\label{1e1}
\end{equation}
where the corresponding parameters will be specified in each example. The associated Ginzburg--Landau energy is defined as follows:
\begin{equation}\label{1e2}
\widetilde{E}_{h}^{n}(\phi,\boldsymbol{u},r)=\int_{\Omega}\left(\frac{1}{2}|\boldsymbol{u}_{h}^{n}|^{2}+\frac{\lambda}{2}|\nabla \phi_{h}^{n}|^{2}\right) d x+\lambda|r_{h}^{n}|^{2}.
\end{equation}

We first present a numerical example to demonstrate the accuracy of the proposed scheme. Since the exact solution is not available, we use the Cauchy error to assess the convergence behavior. Specifically, the error between two successive mesh sizes $h$ and $h/2$ is defined by
$\|e_{\zeta}\|=\|\zeta_{h}-\zeta_{h/2}\|$.

\begin{example}\cite{LS20201}\label{exm0}
	We consider the CHNS system \eqref{1e1} on the domain $\Omega=[0,1]^2$. The initial conditions are chosen as
	\begin{equation}\begin{aligned}
		\phi_{0} &=  \cos(\pi x)\cos(\pi y), \\
		\boldsymbol{u}_{0}&=  \left(-x^{2}(x-1)^{2}(y-1)(2y-1)y/128,y^{2}(y-1)^{2}(x-1)(2x-1)x/128\right)^{T}, \\
		p_{0} &=  0.
	\end{aligned}\end{equation}
	The external force is taken as $\boldsymbol{g}=(0,0)^{T}$. The parameters are given by %$\boldsymbol{g}=\left(\boldsymbol{0,0}\right)^\mathrm{T}$, and the parameters 
    $\varepsilon^{2}=0.1,\ \lambda=0.1,\ M=0.001,\ \ \nu=0.1,\ \ T=0.1,\ \Delta t=1e-4$.
\end{example}
%\end{example}
\begin{figure}[!htbp]
	\centering
	$\begin{array}{c}
		\includegraphics[width=0.64\textwidth]{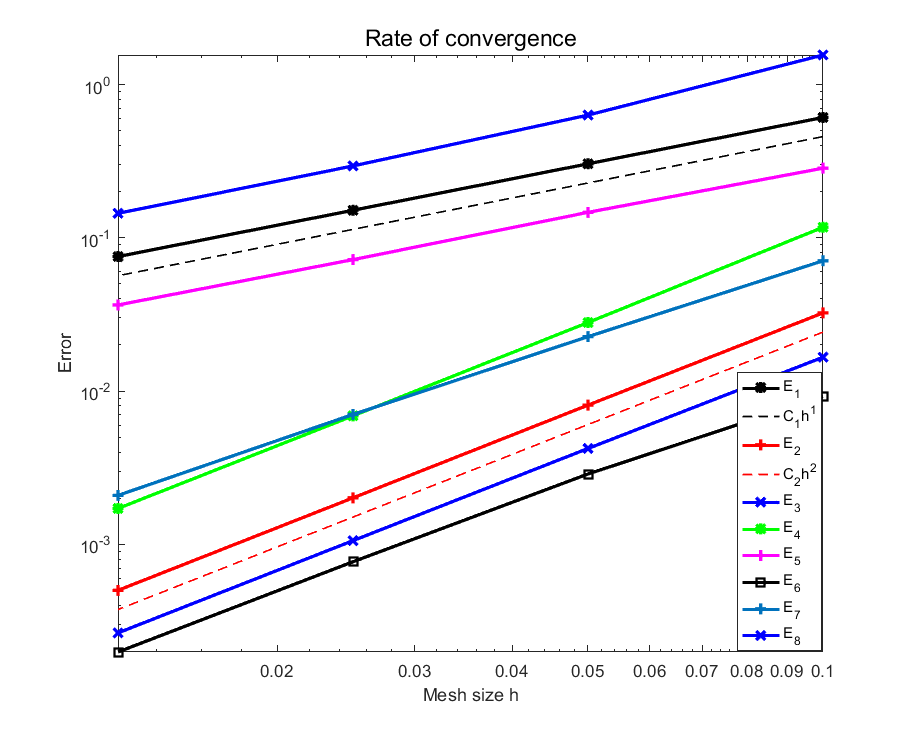}
	\end{array}$
	\caption{Errors and convergent rates.}\label{exp1u1}
\end{figure}

% The errors and convergent rates for spatial discretization between two different grid spacings $h$ and $\frac{h}{2}$ are plotted in Figure \ref{exp1u1}. From the figures, we can see that
The errors and convergence rates for the spatial discretization, obtained using two successive mesh sizes $h$ and $h/2$, are plotted in Figure \ref{exp1u1}. From the figures, we observe that
\[\begin{aligned}
	E_{1}:&=\left\|\phi_{h}^{n+1}-\phi_{h/2}^{n+1}\right\|_{H^{1}}\approx O(h),\qquad\quad
	E_{2}:=\left\|\phi_{h}^{n+1}-\phi_{h/2}^{n+1}\right\|_{L^{2}}\approx O(h^{2}),\\
	E_{3}:&=\left\|\mu_{h}^{n+1}-\mu_{h/2}^{n+1}\right\|_{H^{1}}\approx O(h),\qquad\ \ \
	E_{4}:=\left\|\mu_{h}^{n+1}-\mu_{h/2}^{n+1}\right\|_{L^{2}}\approx O(h^{2}),\\
	E_{5}:&=\left\| \boldsymbol{u}_{h}^{n+1}- \boldsymbol{u}_{h/2}^{n+1}\right\|_{(H^{1})^{2}}\approx O(h),\qquad
	E_{6}:=\left\| \boldsymbol{u}_{h}^{n+1}- \boldsymbol{u}_{h/2}^{n+1}\right\|_{(L^{2})^{2}}\approx O(h^{2}),\\
	E_{7}:&=\left\|p_{h}^{n+1}-p_{h/2}^{n+1}\right\|_{L^{2}}\approx O(h^{2}),\qquad\ \ \ \
	E_{8}:=\left\|r_{h}^{n+1}-r_{h/2}^{n+1}\right\|_{L^{2}}\approx O(h^{2}).
\end{aligned}\]

\begin{example}\cite{F2006}\label{exm1} 
% In the second example, we consider the CHNS equations \eqref{1e1} with the domain $\Omega=[-0.4,0.4]^{2}$, the parameters $\varepsilon=0.01,\ \lambda=1,\ M=0.1,\ \mu=1,\ \nu=1,\ \gamma=1$ and the right term $\boldsymbol{g}=\left(\boldsymbol{1,0}\right)^\mathrm{T}$. For the boundary conditions, we take the homogeneous Dirichlet boundary condition for velocity, and the homogeneous Neumann boundary conditions for $\phi$ and $w$. The initial conditions of velocity and pressure are both set to be zero, and the initial profile of the phase field function is taken as following
% 	\begin{equation}\label{1e3}
% 	\phi_{0}=\tanh\left(\frac{x^{2}}{0.01}+\frac{y^{2}}{0.0225}-1\right).
% 	\end{equation}
    In the second example, we consider the CHNS system \eqref{1e1} on the domain $\Omega=[-0.4,0.4]^2$. The parameters are set to $\varepsilon=0.01$, $\lambda=1$, $M=0.1$, $\nu=1$ and the external force is given by $\mathbf{g}=(1,0)^T$. We impose homogeneous Dirichlet boundary conditions for the velocity and homogeneous Neumann boundary conditions for $\phi$ and $w$. The initial velocity and pressure are both taken to be zero, while the initial phase-field profile is given by
\begin{equation}\label{1e3}
\phi_{0}=\tanh\left(\frac{x^{2}}{0.01}+\frac{y^{2}}{0.0225}-1\right).
\end{equation}
\end{example}

\begin{figure}[!htpb]
\centering
\includegraphics[width=0.325\textwidth]{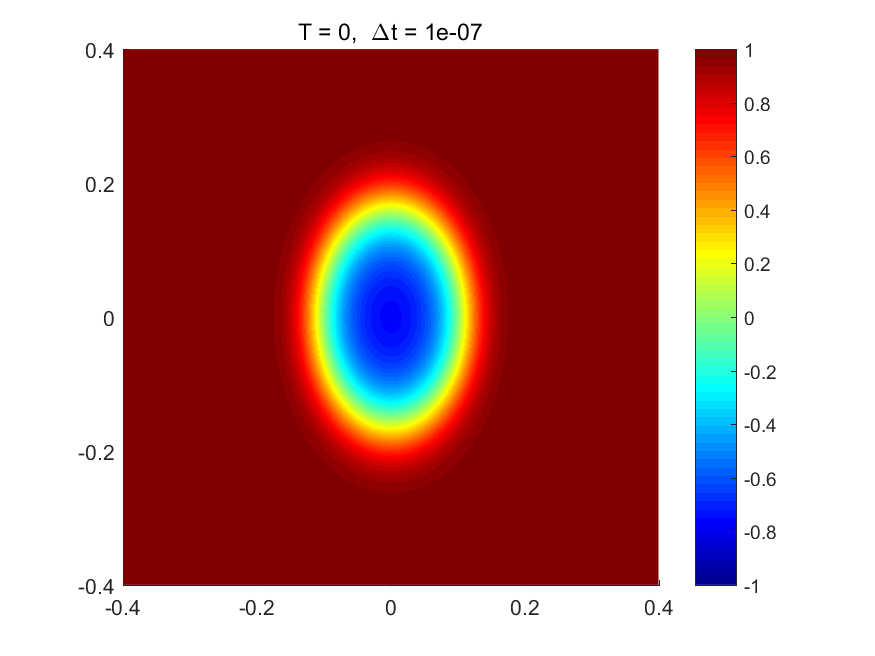}
\hfill
\includegraphics[width=0.325\textwidth]{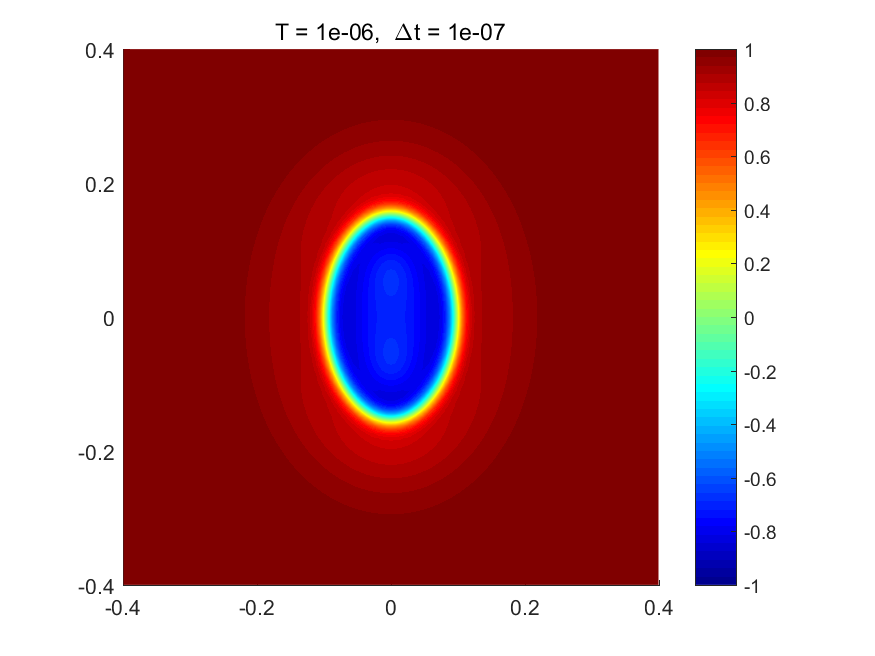}
\hfill
\includegraphics[width=0.325\textwidth]{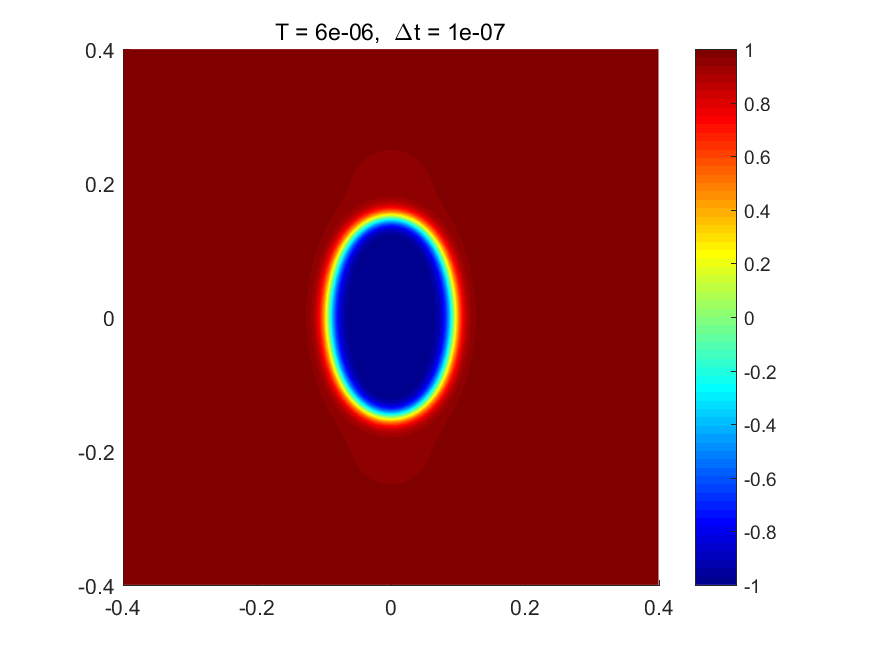}

\vspace{0.3cm} % 控制两行图之间的间距
\includegraphics[width=0.325\textwidth]{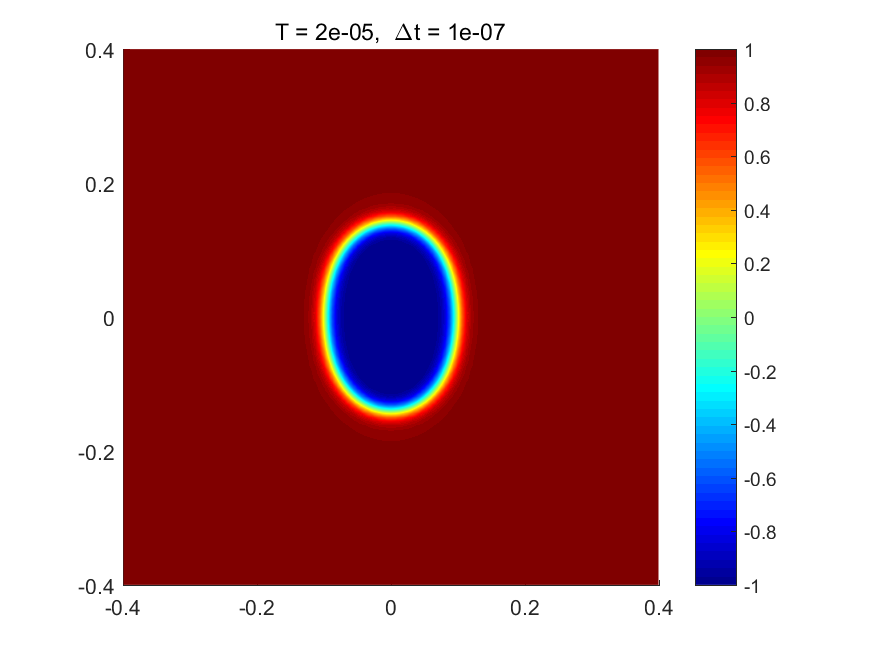}
\hfill
\includegraphics[width=0.325\textwidth]{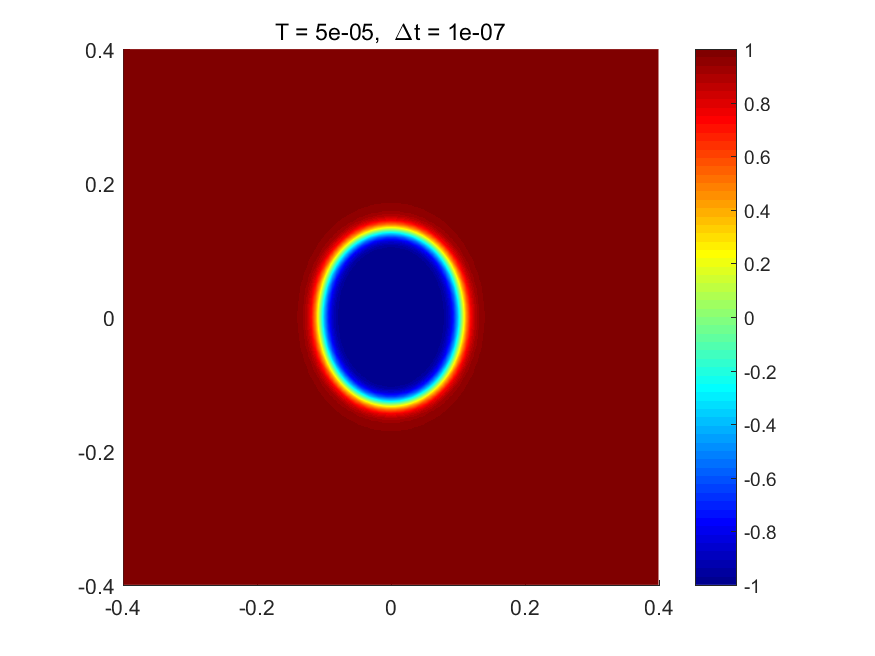}
\hfill
\includegraphics[width=0.325\textwidth]{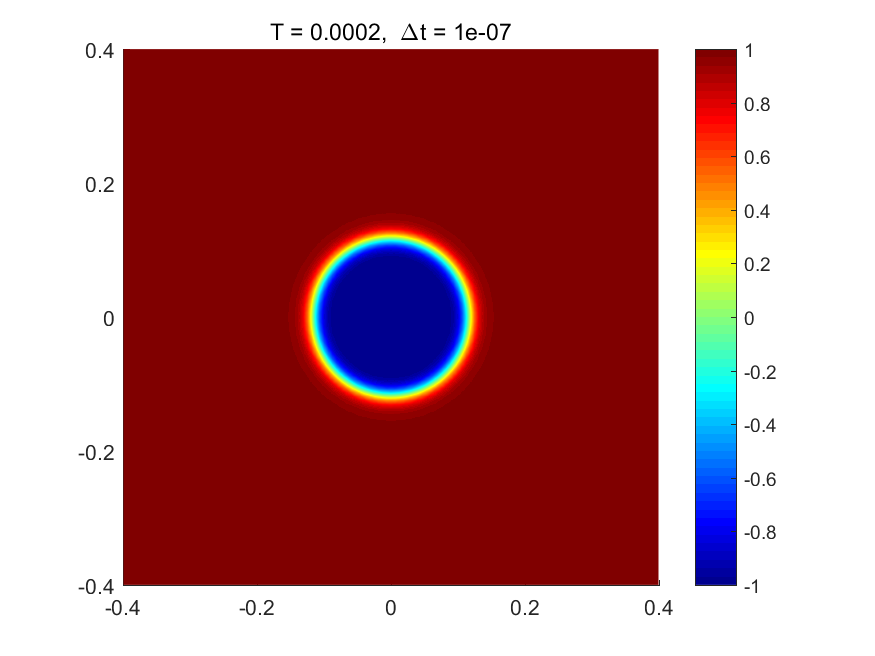}

\caption{Example \ref{exm1}, snapshots of numerical solutions for phase field function.}
\label{exp3u}
\end{figure}
%
% The snapshots of the numerical solutions, which are based on the proposed numerical scheme, for phase field function at six different time steps are shown in {\bf{Figure} \ref{exp3u}} for {\bf{Example} \ref{exm1}}. As it is shown, we can see the initially elliptical interface develops circular shapes gradually.
The snapshots of the numerical solutions obtained using the proposed scheme for the phase-field function at six different time instants are presented in {{Figure} \ref{exp3u}} for {\bf{Example} \ref{exm1}}. As shown in the figures, the initially elliptical interface gradually evolves into a more circular shape over time.

\begin{example}\cite{LS20201}\label{exm2} 
% In this example, we consider the CHNS equations \eqref{1e1} with the domain $\Omega=[0,1]^{2}$, the parameters $\varepsilon=0.01,\ \lambda=0.01,\ M=0.002,\ \mu=1,\ \nu=1,\ \gamma=1$ and the right term $\boldsymbol{g}=\left(\boldsymbol{0,0}\right)^\mathrm{T}$. The initial conditions of velocity and pressure are both set to be zero, and the initial profile of the phase field function is chosen as to be a rectangular bubble, i.e.,
% 	\begin{equation}\label{1e4}
% 	\phi_{0}(x,y)=\left\{\begin{aligned}
% 		1\qquad & if\ 0.25\leq x\leq 0.75  \ \& \ 0.25\leq y\leq 0.75 ,\\
% 		-1\qquad & if\ otherwise.
% 	\end{aligned}
% 	\right.
% 	\end{equation}
% 	For the boundary conditions, we take homogeneous Neumann boundary for $\phi$ and $w$, and boundary condition for velocity is zero.
In this example, we consider the CHNS system \eqref{1e1} on the domain $\Omega=[0,1]^2$. The parameters are chosen as $\varepsilon=0.01$, $\lambda=0.01$, $M=0.002$, $\nu=1$ and the external force is given by $\mathbf{g}=(0,0)^T$. The initial velocity and pressure are both taken to be zero, while the initial phase-field profile is prescribed as a rectangular bubble:
	\begin{equation}\label{1e4}
	\phi_{0}(x,y)=\left\{\begin{aligned}
		1\qquad & if\ 0.25\leq x\leq 0.75  \ \& \ 0.25\leq y\leq 0.75 ,\\
		-1\qquad & if\ otherwise.
	\end{aligned}
	\right.
	\end{equation}
For the boundary conditions, homogeneous Neumann boundary conditions are imposed for $\phi$ and $w$, while homogeneous Dirichlet boundary conditions are applied to the velocity field.
\end{example}

\begin{figure}[!htpb]
\centering
\includegraphics[width=0.325\textwidth]{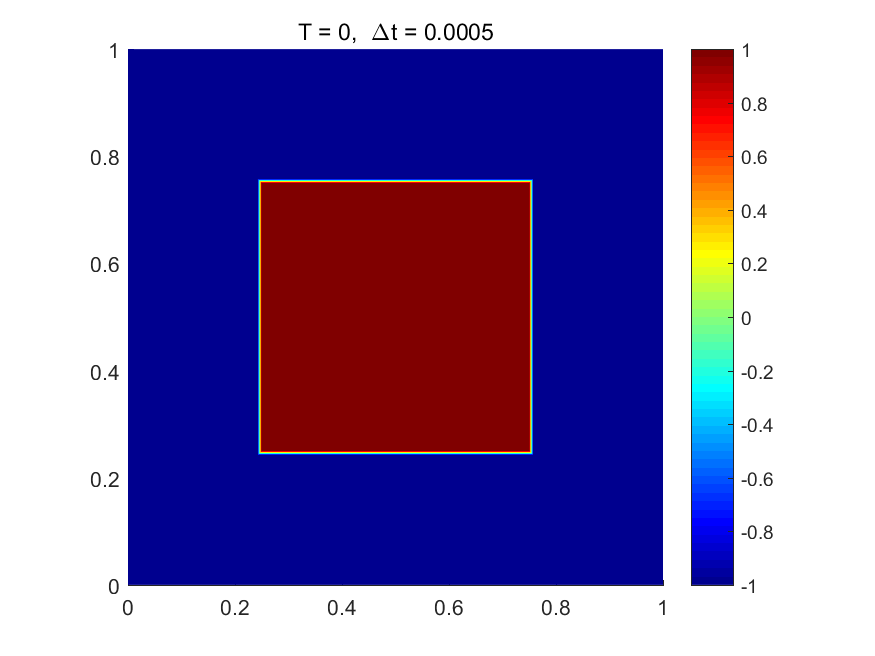}
\hfill
\includegraphics[width=0.325\textwidth]{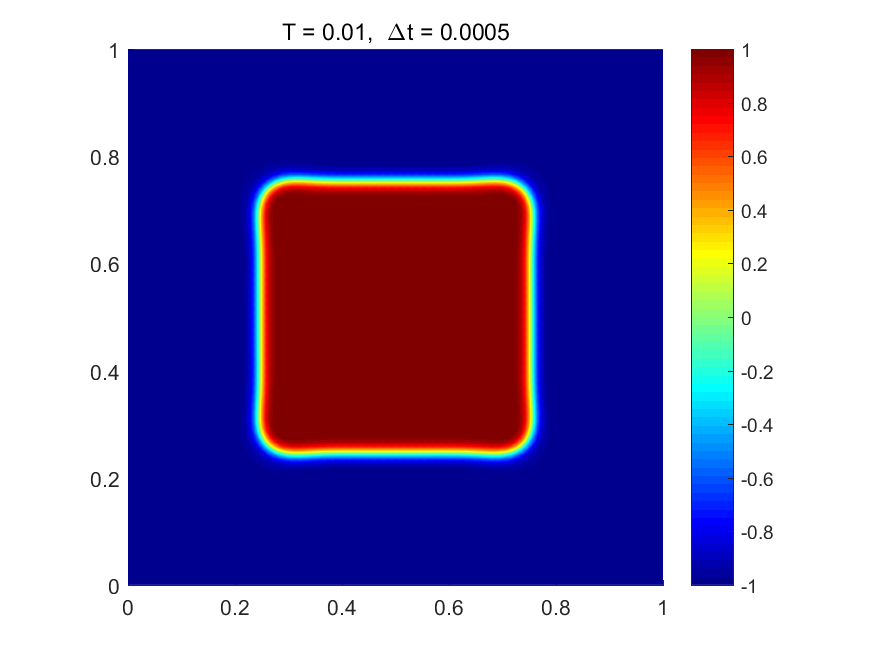}
\hfill
\includegraphics[width=0.325\textwidth]{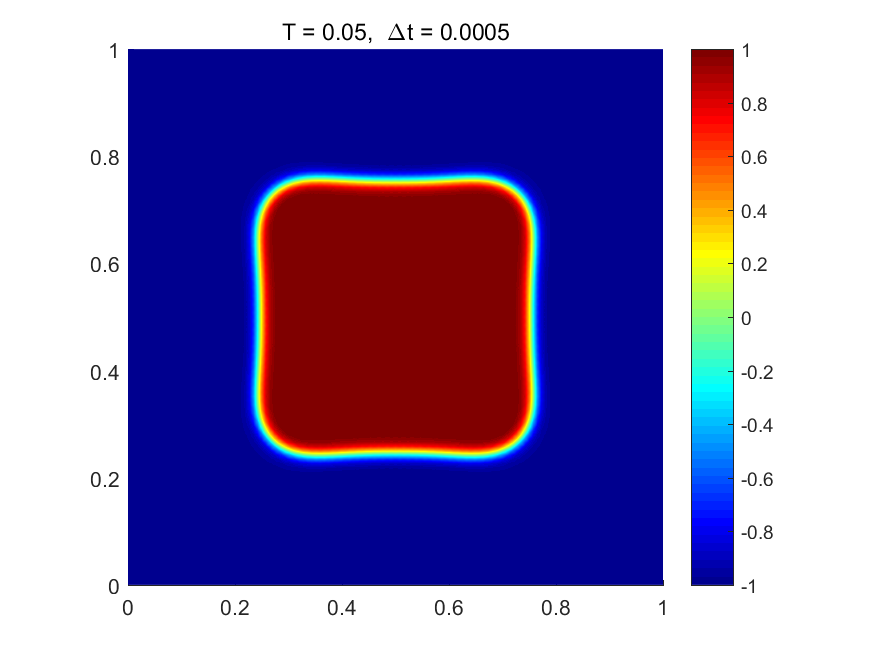}

\vspace{0.3cm} % 行间距
\includegraphics[width=0.325\textwidth]{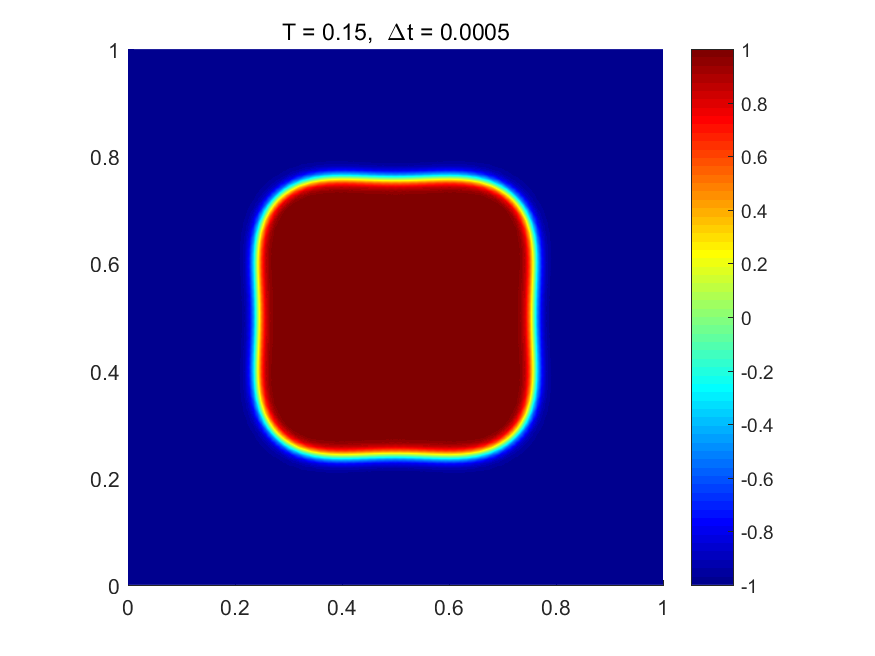}
\hfill
\includegraphics[width=0.325\textwidth]{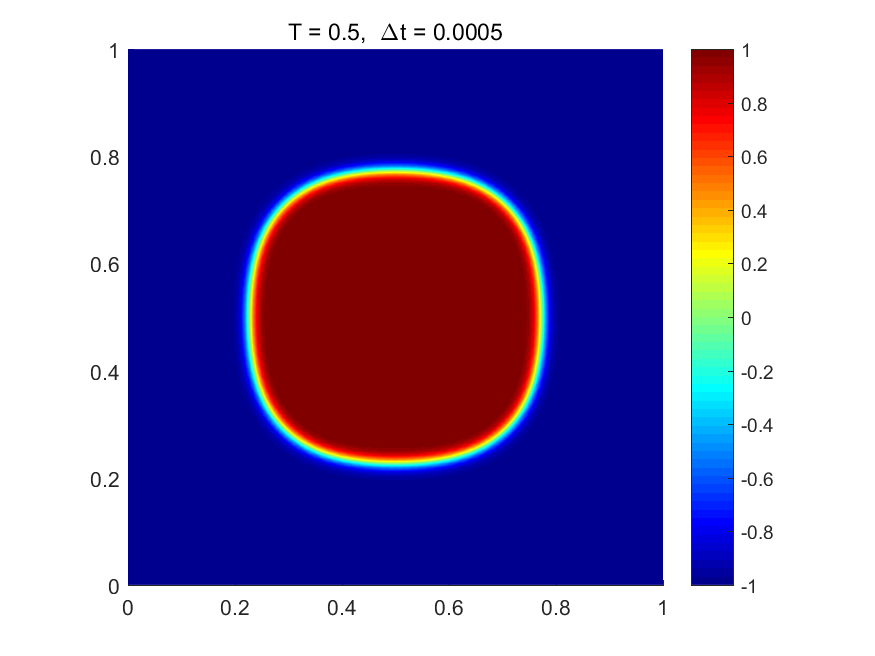}
\hfill
\includegraphics[width=0.325\textwidth]{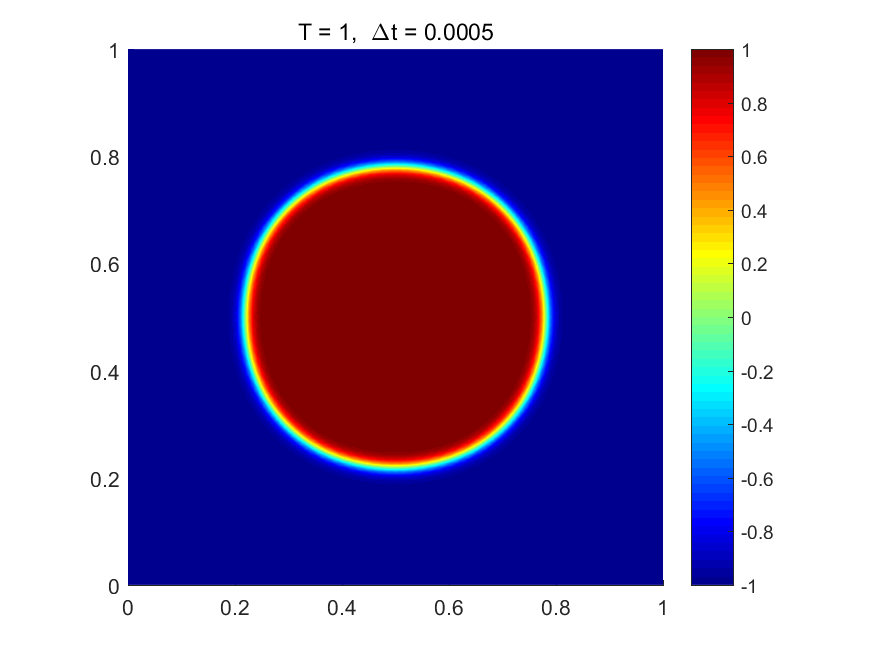}

\caption{Example \ref{exm2}: snapshots of numerical solutions for phase field function.}
\label{exp2u}
\end{figure}

% A sequence snapshots of the numerical solutions for phase field function are produced in view of the proposed scheme in Figure \ref{exp2u} for {\bf{Example} \ref{exm2}}. It can be clearly seen that the interface shape gradually becomes circular over time. 
A sequence of snapshots of the numerical solutions for the phase-field function is presented in Figure \ref{exp2u} for {\bf{Example} \ref{exm2}}, based on the proposed scheme. It can be clearly observed that the interface gradually evolves toward a circular shape over time.

\begin{example}\label{exm4} 
% In the forth example, we consider the CHNS equations \eqref{1e1} with a random phase function in the domain $[-1,1]\times [-1,1]$ to simulate the coarsening dynamics. The initial conditions of velocity and pressure are both set to be zero, and the initial profile of the phase field function is taken as
% 	\begin{equation}\label{1e6}
% 	\phi_{0}(x,y)=rand(x,y)
% 	\end{equation}
% 	with rand(x, y) generates random values between -1 and 1. For the boundary conditions, we take homogeneous Neumann boundary for $\phi$ and $w$, and boundary condition for velocity is zero. And the parameters $\varepsilon=0.02,\ \lambda=0.01,\ M=0.002,\ \mu=1,\ \nu=1,\ \gamma=1$ and the right term $\boldsymbol{g}=\left(\boldsymbol{0,0}\right)^\mathrm{T}$.
In the fourth example, we consider the CHNS equations \eqref{1e1} on the domain $[-1,1]\times[-1,1]$ with a random initial phase-field function to simulate coarsening dynamics. The initial velocity and pressure are both set to zero, while the initial profile of the phase-field function is given by
\begin{equation}\label{1e6}
\phi_{0}(x,y)=\mathrm{rand}(x,y),
\end{equation}
where $\mathrm{rand}(x,y)$ generates random values uniformly distributed in $[-1,1]$. For the boundary conditions, homogeneous Neumann boundary conditions are imposed for $\phi$ and $w$, while no-slip (zero) boundary conditions are prescribed for the velocity field. The parameters are chosen as $\varepsilon=0.02$, $\lambda=0.01$, $M=0.002$, $\nu=1$, and the external force is set to $\boldsymbol{g}=(0,0)^{\mathrm{T}}$.
\end{example}

\begin{figure}[!htpb]
\centering
\includegraphics[width=0.325\textwidth]{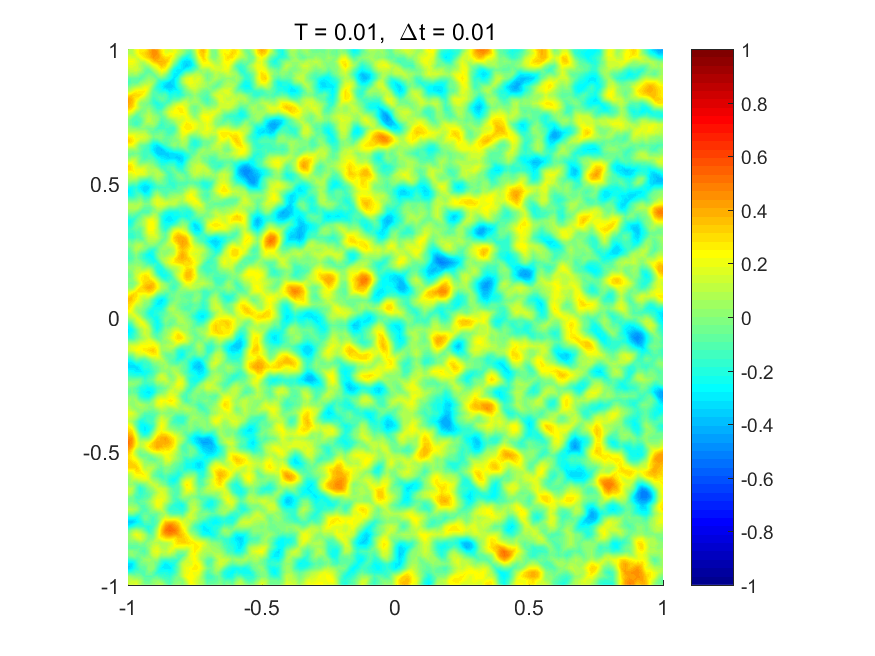}
\hfill
\includegraphics[width=0.325\textwidth]{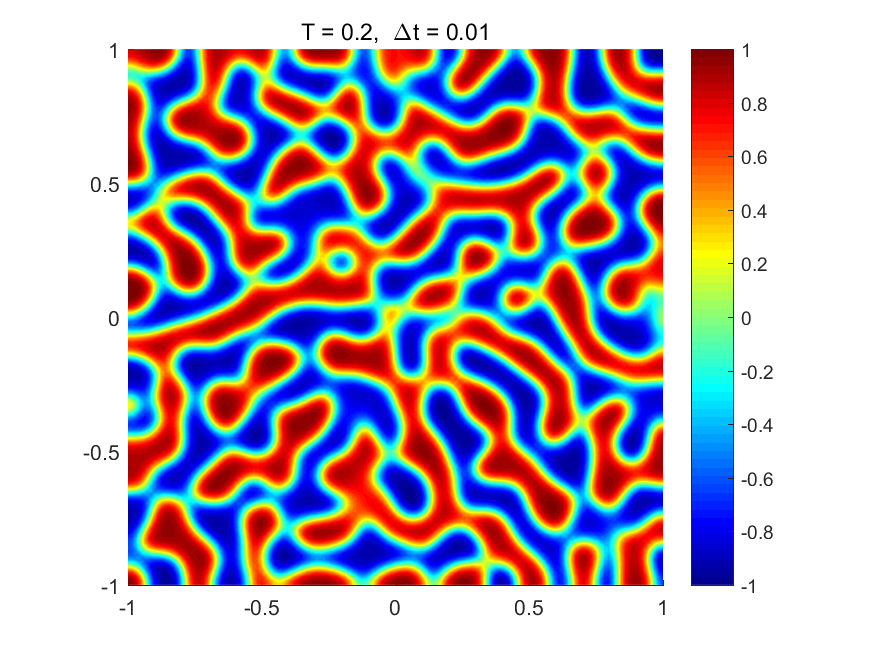}
\hfill
\includegraphics[width=0.325\textwidth]{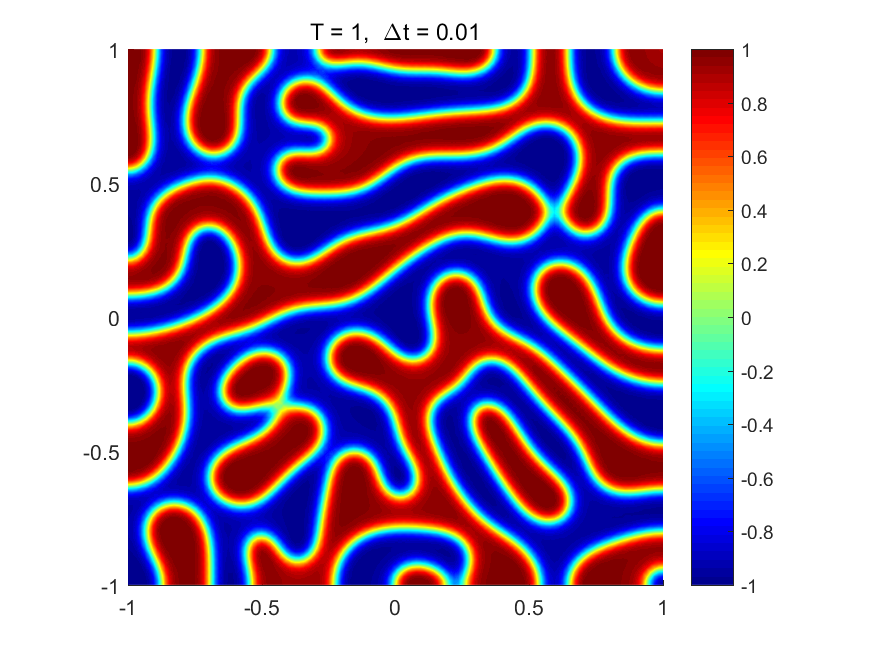}

\vspace{0.3cm} % 控制两行图之间的间距
\includegraphics[width=0.325\textwidth]{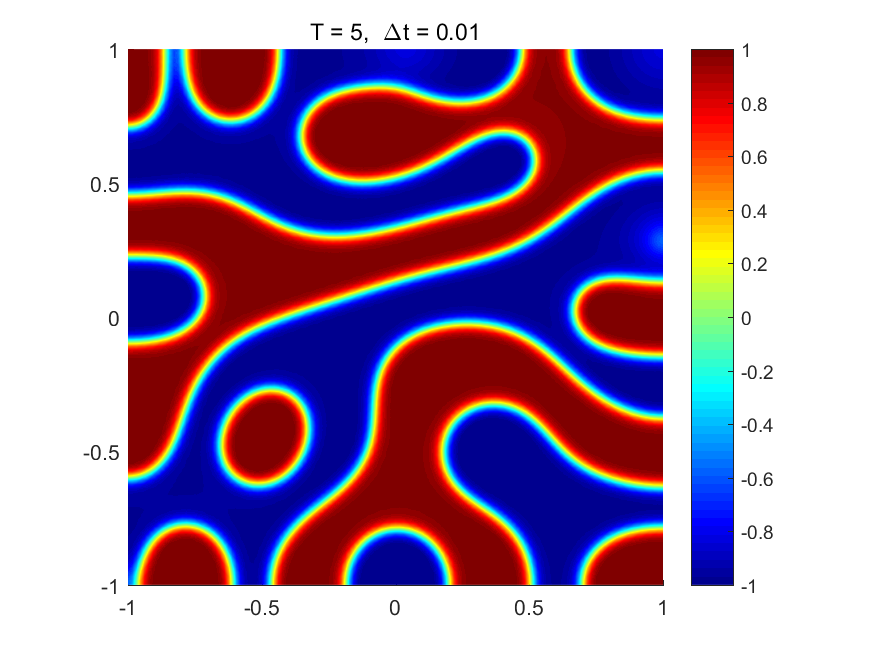}
\hfill
\includegraphics[width=0.325\textwidth]{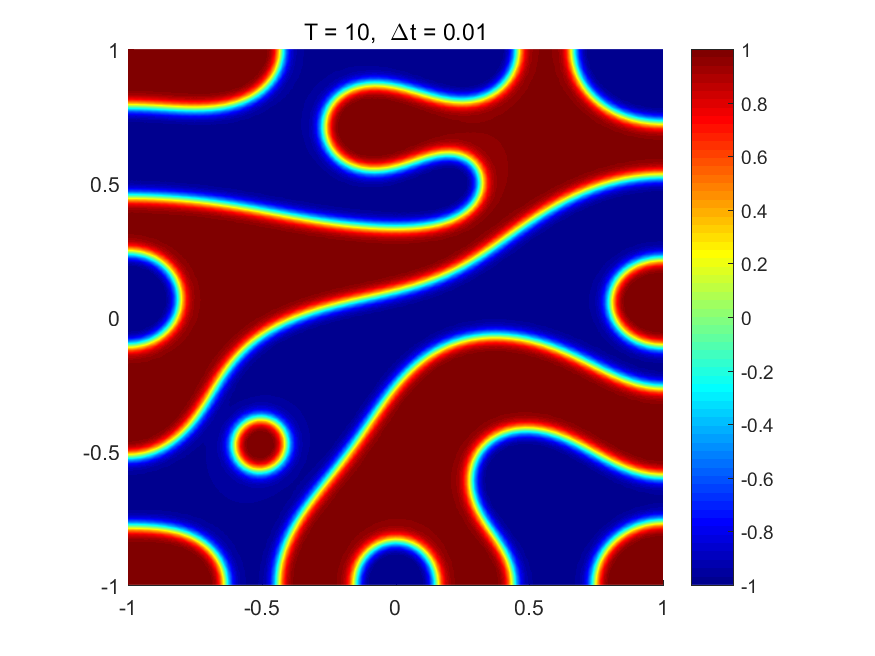}
\hfill
\includegraphics[width=0.325\textwidth]{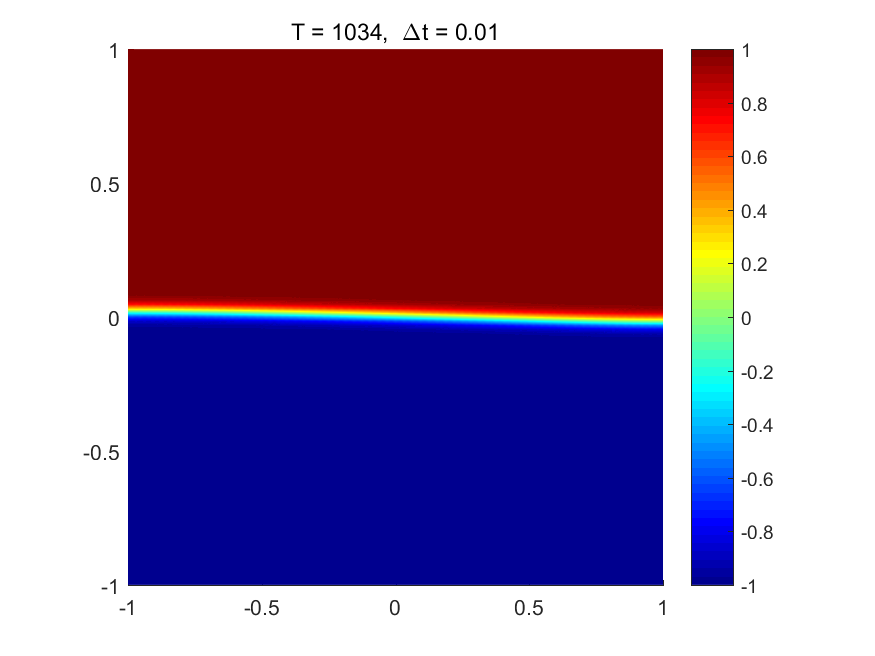}

\caption{Example \ref{exm4}: snapshots of numerical solutions for phase field function.}
\label{exp4u}
\end{figure}

% A sequence snapshots of the numerical solutions for phase field function are produced in view of the proposed scheme in Figure \ref{exp4u} for {\bf{Example} \ref{exm4}}.  We see clearly that the coarsening phenomena.
A sequence of snapshots of the numerical solutions for the phase-field function is presented in Figure \ref{exp4u} for {\bf{Example} \ref{exm4}}, based on the proposed scheme. It can be clearly observed that coarsening phenomena occur over time.

\begin{figure}[!htbp]
\centering
\includegraphics[width=0.325\textwidth]{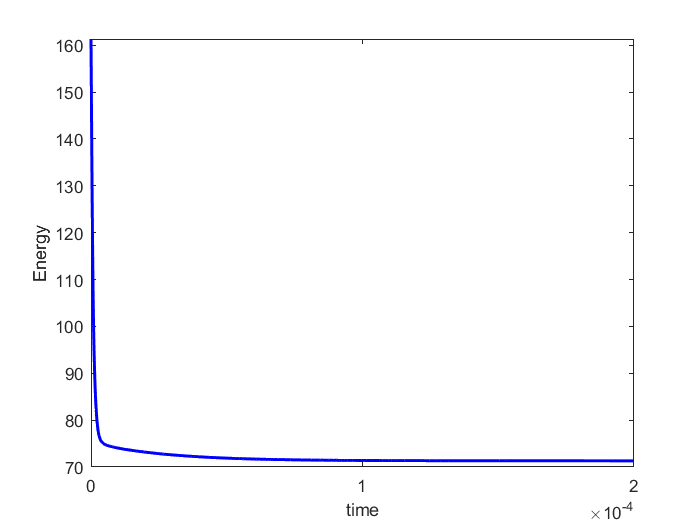}
\hfill
\includegraphics[width=0.325\textwidth]{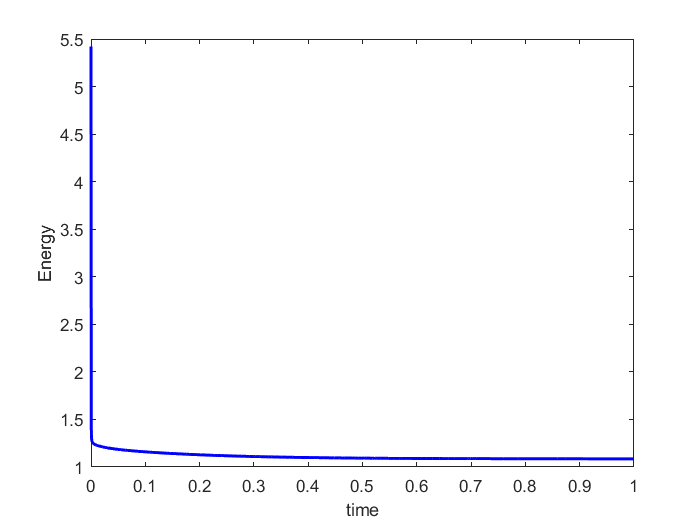}
\hfill
\includegraphics[width=0.325\textwidth]{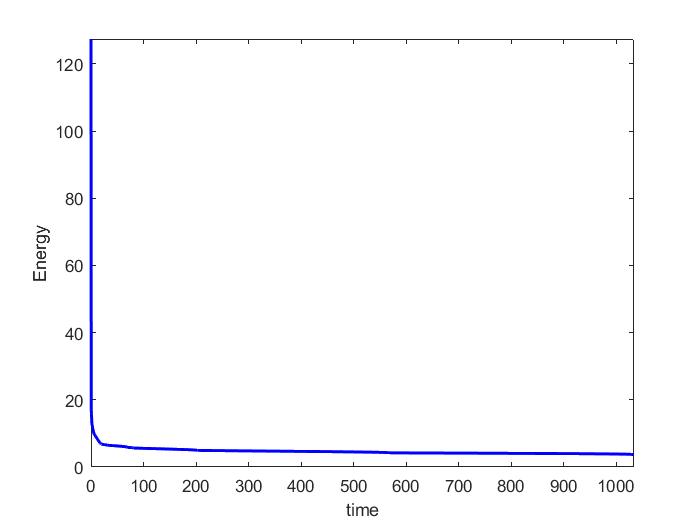}
\caption{The modified discrete energy history,
First line: Example \ref{exm1}; Second line, Left: Example \ref{exm2}; Right: Example \ref{exm4}.}
\label{exp1u0}
\end{figure}

% At last, the time history curves of the modified discrete energy defined in \eqref{1e2} for {\bf{Example} \ref{exm1}-\ref{exm4}} are displayed in {\bf{Figure} \ref{exp1u0}}. It shows that the modified energy decreases on every time step.
Finally, the time-history curves of the modified discrete energy defined in \eqref{1e2} for {\bf{Examples} \ref{exm1}-\ref{exm4}} are presented in Figure \ref{exp1u0}. It can be observed that the modified energy decreases monotonically at each time step.

\section*{Conclusion}\label{secCon}

In this work, we develop a linear, fully decoupled, and unconditionally energy-stable fully discrete finite element scheme for the CHNS system by combining the SAV approach with an IMEX Euler time discretization. Unlike existing SAV-based schemes, which typically require multiple auxiliary variables or additional techniques to handle different nonlinearities, the proposed method introduces only a single scalar auxiliary variable together with a novel update equation for the auxiliary variable. This design yields an efficient decoupled solver structure while preserving the unconditional energy stability of the system.
We established the unconditional discrete energy dissipation law and stability estimates of the scheme, and further derived optimal-order error estimates for the fully discrete finite element approximation. The main difficulty caused by the non-globally Lipschitz nonlinear potential was overcome by proving the uniform boundedness of the numerical solution, which enables the error analysis without additional assumptions on the free energy. Numerical experiments confirmed the theoretical convergence rates and demonstrated the monotonic decay of the modified discrete energy.

Future work will focus on extending the proposed framework to higher-order time discretizations and to more physically realistic CHNS models with variable density and variable viscosity while maintaining unconditional energy stability and efficient decoupling.

\section*{Funding}
J. Yang was supported by the Fundamental Research Funds for the Universities of Henan Province under Grant No. NSFRF2502090, and Natural Science Foundation of Henan Province under Grant No. 262300422624.
N. Yi was supported by the National Key R \& D Program of China (2024YFA1012600) and the NSFC Project (12431014)
P. Yin’s research was supported by the University of Texas at El Paso Startup Award.

\bibliographystyle{abbrv}
\bibliography{ref}

\end{document}